\documentclass[]{amsart}

\usepackage{graphicx,color,hyperref}
\numberwithin{equation}{section}
\renewcommand{\labelenumi}{(\roman{enumi})}
\usepackage[initials,nobysame]{amsrefs}

\newtheorem{theorem}{Theorem}[section]
\newtheorem{proposition}[theorem]{Proposition}

\newtheorem{lemma}[theorem]{Lemma}
\theoremstyle{definition}
\newtheorem{definition}[theorem]{Definition}

\theoremstyle{remark}
\newtheorem{remark}[theorem]{Remark}
\newtheorem{example}[theorem]{Example}
\newtheorem*{acknowledgements}{Acknowledgements}

\DeclareMathOperator{\cn}{cn}

\DeclareMathOperator{\sn}{sn}
\DeclareMathOperator{\am}{am}

\begin{document}

\title{Li--Yau type inequality for curves in any codimension}
\author[T.~Miura]{Tatsuya Miura}
\address[T.~Miura]{Department of Mathematics, Tokyo Institute of Technology, Meguro, Tokyo 152-8511, Japan}
\email{miura@math.titech.ac.jp}
\keywords{Li--Yau inequality, embeddedness, elastica, multiplicity, elastic flow, elastic network}
\subjclass[2020]{53A04, 49Q10, and 53E40}

\begin{abstract}
  For immersed curves in Euclidean space of any codimension we establish a Li--Yau type inequality that gives a lower bound of the (normalized) bending energy in terms of multiplicity.
  The obtained inequality is optimal for any codimension and any multiplicity except for the case of planar closed curves with odd multiplicity; in this remaining case we discover a hidden algebraic obstruction and indeed prove an exhaustive non-optimality result.
  The proof is mainly variational and involves Langer--Singer's classification of elasticae and Andr\'{e}'s algebraic-independence theorem for certain hypergeometric functions.
  We also discuss applications to elastic flows, networks, and knots.
\end{abstract}

\maketitle

%\setcounter{tocdepth}{1}
%\tableofcontents

\section{Introduction}\label{sectintroduction}

The classical Li--Yau inequality \cite{Li1982} asserts that if a closed surface $\Sigma\subset\mathbf{R}^n$, $n\geq3$, has a point of multiplicity $k\geq1$, then the Willmore energy $W[\Sigma]:=\int_{\Sigma}|H|^2dS$ is bounded below by multiplicity in the form of
\begin{equation}\label{eq:Li-Yau_Willmore}
  W[\Sigma]\geq4\pi k,
\end{equation}
where $H$ denotes the mean curvature vector.
(See also a different proof in $\mathbf{R}^3$ \cite{Topping1998}.)
This estimate is sharp due to a nearly $k$-times covered sphere.
In particular, if $W[\Sigma]<8\pi$ then $\Sigma$ must be embedded.
This result is used as a fundamental tool in many studies; the Willmore flow \cite{Kuwert2004}, the Willmore conjecture \cite{Marques2014}, and others.

In this paper we establish a one-dimensional analogue of the Li--Yau inequality, and reveal that a new phenomenon emerges due to low dimensionality.
For an immersed curve $\gamma$ in $\mathbf{R}^n$ we let $\kappa$ denote the curvature vector $\kappa:=\partial_s^2\gamma$, where $\partial_s\psi:=\frac{1}{|\gamma'|}\psi'$, and define the {\em normalized bending energy} $\bar{B}[\gamma]$ as the bending energy $B[\gamma]:=\int_\gamma|\kappa|^2ds$ normalized by the length $L=L[\gamma]:=\int_\gamma ds$ to be scale-invariant:
\begin{equation*}
  \bar{B}[\gamma]:=L[\gamma]B[\gamma]=L\int_\gamma|\kappa|^2ds.
\end{equation*}
In addition, using the complete elliptic integral of the first kind $K(m)$ and of the second kind $E(m)$, we define a unique parameter $m^*\in(0,1)$ such that $K(m^*)=2E(m^*)$, and then the key universal constant $\varpi^*>0$ by
\begin{equation}\label{eq:constant}
  \varpi^*:= 32(2m^*-1)E(m^*)^2\ (=28.109...)
\end{equation}
Finally, we say that a curve $\gamma$ has a point $p\in\mathbf{R}^n$ of {\em multiplicity} $k$ if the preimage $\gamma^{-1}(p)$ contains at least $k$ distinct points.

Our first theorem asserts a general Li--Yau type inequality involving multiplicity for closed curves $\gamma:\mathbf{T}^1\to\mathbf{R}^n$, where $\mathbf{T}^1:=\mathbf{R}/\mathbf{Z}$.
Hereafter we specify the natural $H^2$-Sobolev regularity for curves.

\begin{theorem}[Multiplicity inequality for closed curves]\label{thm:Li-Yau_closed}
  Let $n\geq2$ and $k\geq2$.
  Let $\gamma:\mathbf{T}^1\to\mathbf{R}^n$ be an immersed closed $H^2$-curve with a point of multiplicity $k$.
  Then
  \begin{equation}\label{eq:Li-Yau_closed}
    \bar{B}[\gamma]\geq \varpi^*k^2.
  \end{equation}
\end{theorem}

In particular, if an immersed closed curve $\gamma$ has the property that $\bar{B}[\gamma]<4\varpi^*$, then $\gamma$ must be embedded.
This threshold is optimal because a figure-eight elastica gives an explicit example of a non-embedded analytic planar closed curve with energy $\bar{B}=4\varpi^*$ (see Definition \ref{def:figureeight} and Lemma \ref{lem:figureeight}).

We also discuss more on optimality and rigidity in inequality \eqref{eq:Li-Yau_closed}.%, where we specify the natural $H^2$-Sobolev regularity for curves.
On one hand, our inequality is optimal for many pairs of $(n,k)$, namely either if $n\geq3$ or if $k$ is even.
We also prove the rigidity that any optimal curve is a {\em $k$-leafed elastica}, i.e., the curve consists of $k$ half-fold figure-eight elasticae of same length (see Definition \ref{def:leafedelastica}).

\begin{theorem}[Optimality and rigidity]\label{thm:rigidity_closed}
  Let $n\geq2$ and $k\geq2$.
  Suppose either that $n\geq3$ or that $k$ is even.
  Then there exists an immersed closed $H^2$-curve $\gamma:\mathbf{T}^1\to\mathbf{R}^n$ with a point of multiplicity $k$ such that
  \begin{equation}\label{eq:rigidity_closed}
    \bar{B}[\gamma]= \varpi^*k^2.
  \end{equation}
  In addition, equality \eqref{eq:rigidity_closed} is attained if and only if $\gamma$ is a closed $k$-leafed elastica.
\end{theorem}

In particular, any $2$-leafed elastica is (up to invariances) uniquely given by a figure-eight elastica, which is analytic and planar.
Any $3$-leafed is uniquely given by a new three-dimensional shape, which we introduce in Example \ref{ex:propeller} and call \emph{elastic propeller}, whose regularity is of class $C^{2,1}=W^{3,\infty}$ but not $C^3$.
For $k\geq4$, leafed elasticae are generically nonunique and not $C^3$.
See Section \ref{sect:rigidity} for details.

On the other hand, somewhat interestingly, in the remaining case of $n=2$ and odd $k\geq3$ (planar closed curves with odd multiplicity) a new algebraic obstruction comes into play and indeed we can prove an exhaustive non-optimality result.

\begin{theorem}[Non-optimality]\label{thm:nonoptimality}
  For any odd integer $k\geq3$ there exists a positive number $\varepsilon_k>0$ such that for any immersed (planar) closed $H^2$-curve $\gamma:\mathbf{T}^1\to\mathbf{R}^2$ with a point of multiplicity $k$,
  $$\bar{B}[\gamma] \geq \varpi^*k^2 + \varepsilon_k.$$
\end{theorem}

Our results are new for all $n\geq3$ or $k\geq3$.
In their very recent study \cite{Mueller2021}, M\"{u}ller--Rupp obtain Theorems \ref{thm:Li-Yau_closed} and \ref{thm:rigidity_closed} for the special pair $(n,k)=(2,2)$, finding the non-simple threshold $4\varpi^*$ ($=:c^*=112.439...$ in their notation).
They crucially use the assumption that $(n,k)=(2,2)$ since their proof relies on the fact that any planar closed curve with rotation number $\neq\pm1$ has a self-intersection; in particular, they explicitly mention that the case of $n\geq3$ is remained open.
Our results resolve this problem, while retrieving their result as a special case by a different (in fact shorter) proof.
For a general multiplicity $k$ the non-optimal estimate $\bar{B}\geq16k^2$ was previously obtained by several authors \cite[Corollary 3.3.1.3]{Polden1996}, \cite[Theorem 4.4]{Mosel1998}, \cite[Theorem 1.6]{Wheeler2013} (see also \cite[Lemma 2.1]{Wojtowytsch2021}).
All those results are improved by Theorem \ref{thm:rigidity_closed} and optimized in many cases.
Theorem \ref{thm:nonoptimality} highlights a new phenomenon compared to the original Li--Yau inequality \eqref{eq:Li-Yau_Willmore}, which instead is sharp regardless of codimension and multiplicity.

The study of the bending energy $B$ was initiated by D.\ Bernoulli and L.\ Euler in the 18th century for modelling planar elastic rods, but is still ongoing; see e.g.\ \cite{Love1944,Truesdell1983,Sachkov2008,Miura2020} and references therein.
Corresponding variational solutions are called elastic curves or elasticae, and their known classification plays a key role in our study.
Our results have direct applications to more modern subjects such as elastic flows and elastic networks.
In Section \ref{sect:elasticflow}, we apply our inequality to obtain optimal energy thresholds below which %all-time embeddedness holds along elastic flows
any elastic flow must be embedded for all time $t\geq0$, in the same manner as \cite{Mueller2021}; see Theorems \ref{thm:elasticflow} and \ref{thm:fixedlengthelasticflow}.
In Section \ref{sect:elasticnetwork} we are concerned with existence of minimal elastic $\Theta$-networks.
This problem was solved by Dall'Acqua--Novaga--Pluda in the planar case \cite[Theorem 4.10]{DallAcqua2020} (and \cite{DallAcqua2021}), see also \cite{DallAcqua2017}, but they indicate in the last paragraph of \cite[Section 4]{DallAcqua2020} that it remains open in higher codimensions.
Here we resolve this problem in Theorem \ref{thm:existence_network}.

The normalized bending energy $\bar{B}$ is a natural one-dimensional counterpart of the Willmore energy in the sense that both are scale-invariant functionals involving curvature and minimized only by a round shape.
The total curvature $TC[\gamma]:=\int_\gamma|\kappa|ds$ is also similar but not effective for detecting embeddedness of closed curves since both the infima among embedded and non-embedded closed curves coincide with $2\pi$; see \cite{Mueller2021}.
The energy $\bar{B}$ is certainly effective since a circle attains $\bar{B}=4\pi^2<4\varpi^*$.
Recall that $4\pi^2$ is the minimum of $\bar{B}$ among closed curves since $\bar{B}\geq TC^2\geq 4\pi^2$ holds by the Cauchy-Schwarz inequality and Fenchel's theorem.

We now discuss the idea of our proof.
To apply variational methods we encounter the multiplicity-constraint making the admissible set non-open.
M\"{u}ller--Rupp's proof \cite{Mueller2021} is mainly devoted to a careful analysis of possible self-intersections by using the rotation number, which has no direct extension to higher codimensions or multiplicities.
Instead, our proof proceeds in such a way that we divide the objective curve at the point of multiplicity and then apply a variational argument to each component ``independently''.
Each variational problem is formulated to be well posed, and moreover its boundary condition is relaxed to being of zeroth order (although the most natural choice would be of first order since $H^2\hookrightarrow C^1$).
This relaxation allows us to obtain a strong rigidity of optimal configurations, Proposition \ref{prop:minimizer}, which benefits from the celebrated classification of elasticae by Langer--Singer \cite{Langer1984a}.
It is somewhat by chance that such independent and relaxed problems can be translated back to the original problem (before division) while keeping certain optimality.
Indeed, nontriviality of this point is explicitly reflected in our non-optimality result, Theorem \ref{thm:nonoptimality}.
The non-optimality is caused by an obstruction for constructing optimal planar closed curves, which is related with the irrationality of a certain geometric quantity.
Although such an issue is quite delicate in general, surprisingly at least to the author, we can exhaustively verify non-optimality by reducing the problem to a classical deep result of Andr\'{e} \cite{Andre1996} on the algebraic independence of values of certain hypergeometric functions %(directly related with complete elliptic integrals)
over the field of algebraic numbers.
The case of higher codimensions stands in stark contrast to the planar case as it allows unified optimality, Theorem \ref{thm:rigidity_closed}.
The main ingredient here is the aforementioned elastic propeller in Example \ref{ex:propeller}.

The above dividing idea motivates us to consider not only closed curves but also open curves.
In fact, we mainly deal with open curves in our proof, and obtain very parallel results to Theorems \ref{thm:Li-Yau_closed} and \ref{thm:rigidity_closed}, see Theorems \ref{thm:Li-Yau_open} and \ref{thm:rigidity_open}.
As for open curves, our results are fully optimal for all codimensions and multiplicities.

Finally, we indicate that the elastic propeller would be of particular interest in view of elastic knot theory, cf.\ \cite{Gerlach2017} and references therein.
More precisely, the elastic unknot of smallest energy is trivially the one-fold circle, but some numerical studies \cite{Avvakumov2014,Bartels2021} suggest another possible ``stable'' elastic unknot, which looks like a propeller and is experimentally reproducible by a springy wire as in Figure \ref{fig:propeller_photo}.
Our elastic propeller would give the first analytic representation of such a stable shape.
Our result already partially supports this conjectural stability as it implies minimality among all perturbations keeping the triple point.
In fact, we expect that the elastic propeller would be the stable elastic unknot of second smallest energy, partially because the figure-eight elastica is unstable in the space of a suitable closure of the trivial knot class.

\begin{figure}[htbp]
  \includegraphics[width=50mm]{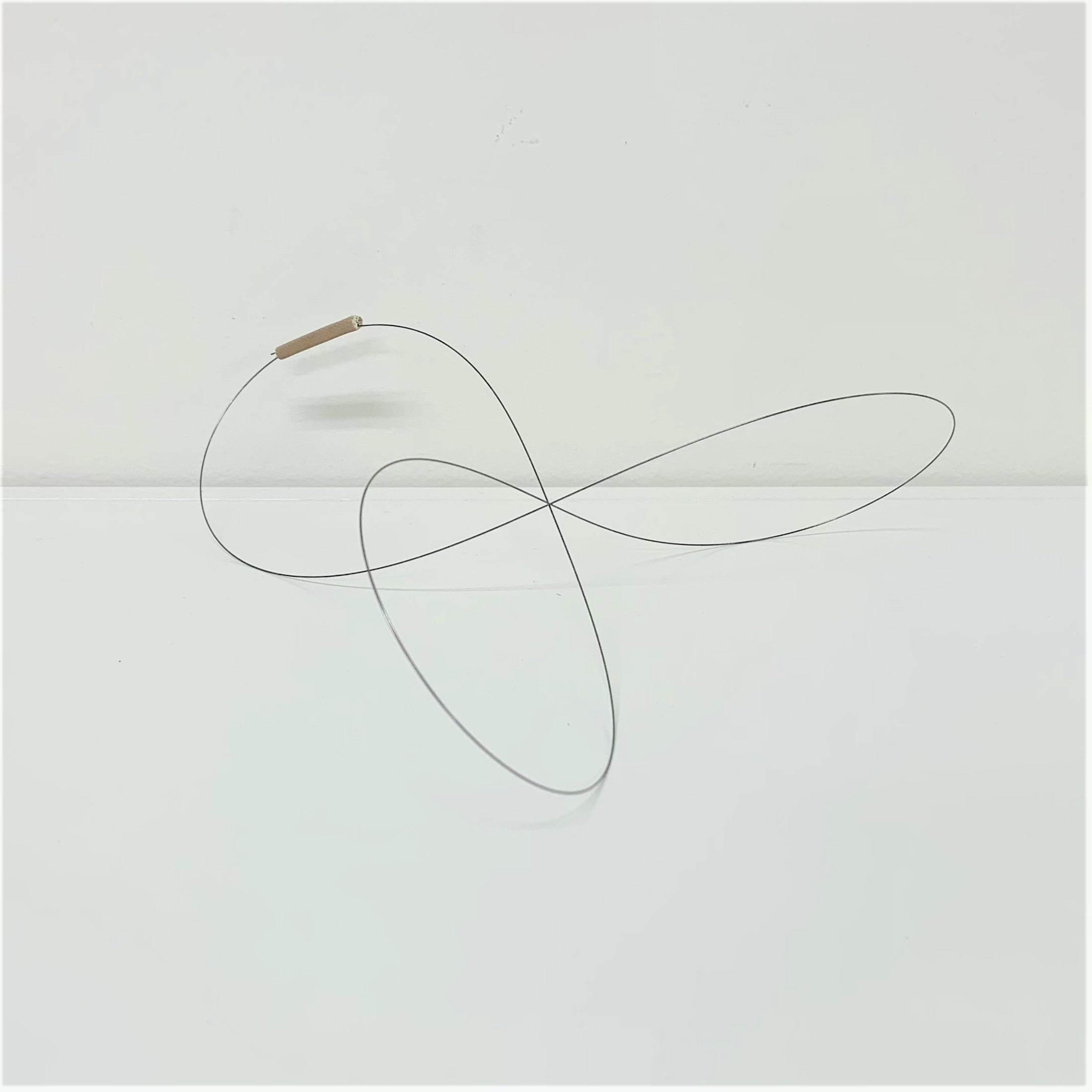}
  \qquad
  \includegraphics[width=50mm]{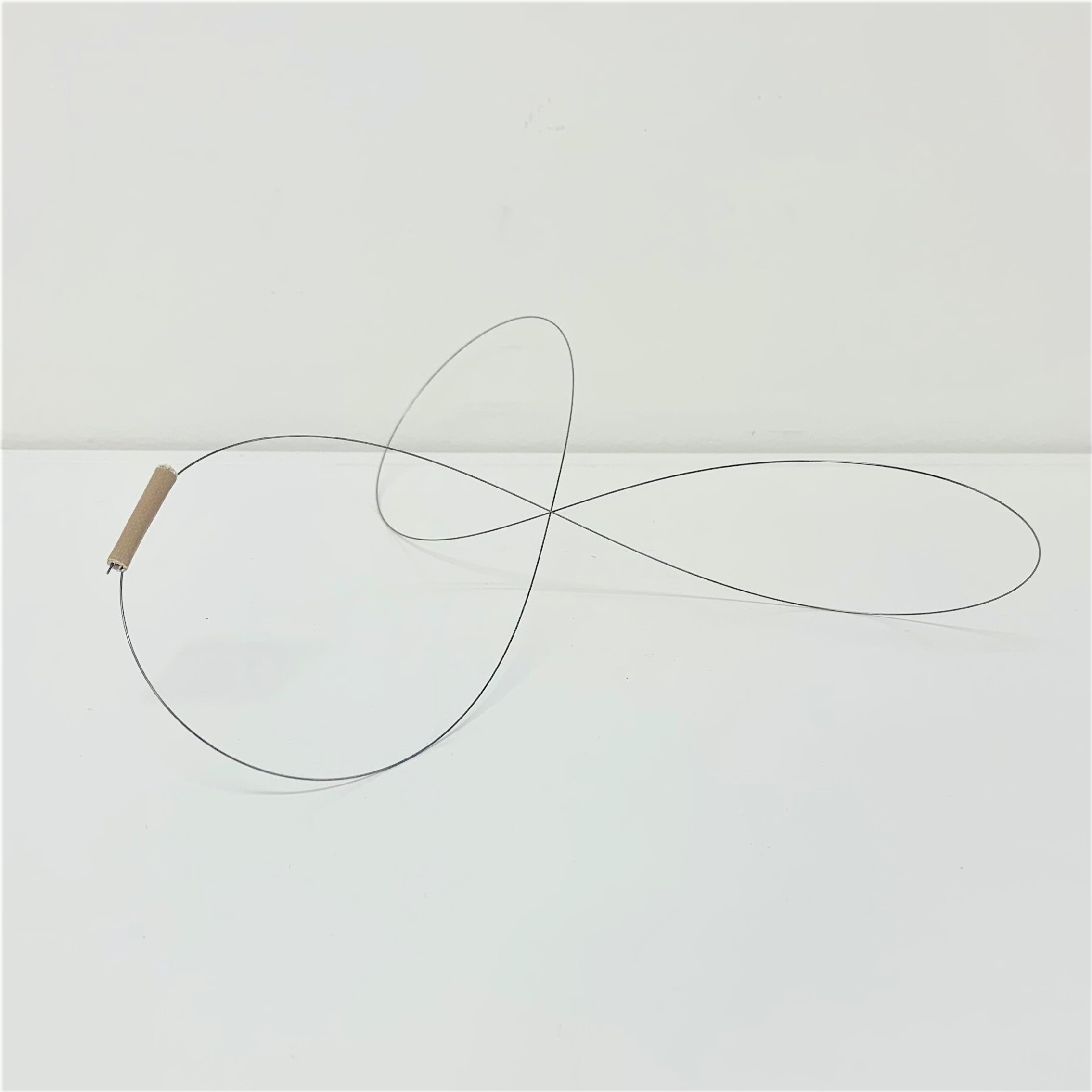}
  \caption{Propeller made of unknotted wire.}
  \label{fig:propeller_photo}
\end{figure}

This paper is organized as follows:
In Section \ref{sect:Li-Yau} we recall and discuss classical elasticae and prove Theorem \ref{thm:Li-Yau_closed} via the open-curve counterpart, Theorem \ref{thm:Li-Yau_open}.
In Section \ref{sect:rigidity} we introduce leafed elasticae and mainly discuss their rigidity, which is then applied to the proof of Theorems \ref{thm:rigidity_closed} and \ref{thm:nonoptimality} again via the open-curve counterpart, Theorem \ref{thm:rigidity_open}.
Sections \ref{sect:elasticflow} and \ref{sect:elasticnetwork} are about applications to elastic flows and elastic networks, respectively.

%\subsection*{Acknowledgements}
\begin{acknowledgements}
  The author would like to thank Marius M\"{u}ller, Fabian Rupp, and Ryotaro Sakamoto for their helpful comments and discussions.
  %He would also like to thank the anonymous referee for the comments that improve the readability.
  This work is in part supported by JSPS KAKENHI Grant Numbers 18H03670, 20K14341, and 21H00990, and by Grant for Basic Science Research Projects from The Sumitomo Foundation.
\end{acknowledgements}

\section{Elastica and Li--Yau type multiplicity inequality}\label{sect:Li-Yau}

The goal of this section is to prove Theorem \ref{thm:Li-Yau_closed}.
To this end we review and prove some results concerning classical elasticae; the most essential step is Proposition \ref{prop:minimizer}.
In particular, the so-called figure-eight elastica plays a key role throughout in this paper.
To define this we need to use some properties of elliptic integrals and functions, which we first address below for the sake of logical order.

\subsection{Elliptic integrals and functions}\label{subsect:elliptic}

Here we briefly collect some facts about Jacobi elliptic integrals and functions.
For more details see classical textbooks, e.g.\ \cite[Chapter XXII]{Whittaker1962} (and also \cite{Abramowitz1992}).

The {\em incomplete elliptic integral of the first kind} $F(x,m)$ and {\em of the second kind} $E(x,m)$ with {\em parameter} $m\in(0,1)$ (squared elliptic modulus) are defined by
\begin{equation*}
  F(x,m):=\int_0^x\frac{d\theta}{\sqrt{1-m\sin^2\theta}}, \quad E(x,m):=\int_0^x\sqrt{1-m\sin^2\theta}d\theta,
\end{equation*}
respectively.
The {\em complete elliptic integral of first kind} $K(m)$ and {\em of second kind} $E(m)$ are then defined by
$$K(m):=F(\pi/2,m), \quad E(m):=E(\pi/2,m),$$
respectively.
The {\em (Jacobi) amplitude function} is defined by
$$\am(\cdot,m):=F^{-1}(\cdot,m) \quad \text{on}\ \mathbf{R}.$$
The {\em (Jacobi) elliptic functions} are then given by
\begin{align*}
  \cn(x,m)&:=\cos(\am(x,m)),\quad \sn(x,m):=\sin(\am(x,m)).
  %\dn(x,m)&:=\sqrt{1-m\sin^2(\am(x,m))}.
\end{align*}
Note in particular that $\cn(\cdot,m)$ and $\sn(\cdot,m)$ are $4K(m)$-periodic, have zeroes $(2\mathbf{Z}+1)K(m)$ and $2\mathbf{Z}K(m)$, respectively, and change their sign at the zeroes (like cosine and sine).

To define a figure-eight elastica we need to define a unique parameter $m^*\in(0,1)$ such that $K(m^*)=2E(m^*)$; numerically, $m^*\approx0.82611$.
Such a parameter indeed exists uniquely since it is easy to check that the continuous function $K(m)-2E(m)$ is increasing from $-\pi/2$ to $\infty$.
%by definition of $K$ (resp.\ $E$) it is obvious that $K$ is increasing (resp.\ $E$ is decreasing) in $m\in(0,1)$ and that $K(0)=E(0)=\pi/2$, $K(1-0)=\infty$, and $E(1)=1$; hence, the continuous function $K(m)-2E(m)$ is increasing from $-\pi/2$ to $\infty$, so that
In addition, one can also easily check that
\begin{equation}\label{eq:parameter0.5}
  m^*>0.5
\end{equation}
through the negativity of the integrand of $K(\frac{1}{2})-2E(\frac{1}{2})$.
This is enough sharp for our argument in this section, but later we need to improve this estimate, cf.\ Lemma \ref{lem:parameter} below.

\subsection{Classical elastica}\label{subsect:elastica}

Here and hereafter $n\geq2$ is arbitrary if not specified.
The classical Lagrange multiplier method ensures that if a smooth curve $\gamma:[a,b]\to\mathbf{R}^n$ minimizes the bending energy $B$ in a suitable class of fixed-length curves, then there is some $\lambda\in\mathbf{R}$ such that the curve $\gamma$ is a critical point of the energy
\begin{equation}\label{eq:energy_E_lambda}
  E_\lambda[\gamma]:=B[\gamma]+\lambda L[\gamma] = \int_\gamma(|\kappa|^2+\lambda)ds.
\end{equation}
By calculating the first variation of $E_\lambda$ (cf.\ \cite{Dziuk2002,DallAcqua2014}) we obtain a fourth-order ODE,
\begin{equation}\label{eq:elastica}
  2\nabla_s^2\kappa+|\kappa|^2\kappa-\lambda\kappa=0,
\end{equation}
where $\nabla_s$ denotes the normal derivative along $\gamma$; more precisely, $\nabla_s\psi:=\partial_s\psi-\langle\partial_s\psi,\partial_s\gamma\rangle \partial_s\gamma$, and here and hereafter $\langle\cdot,\cdot\rangle$ denotes the Euclidean inner product.
%A more tractable form of the equation in terms of curvature and torsion is also found in \cite{Griffiths1983,Langer1984a}.

A smooth curve $\gamma$ that solves \eqref{eq:elastica} for some $\lambda\in\mathbf{R}$ is called {\em elastica}.
The classification of planar elasticae is essentially obtained by Euler in 1744 and thus very classically known (cf.\ \cite{Linner1996,Sachkov2008}).
Concerning general elasticae, Langer--Singer's landmark study \cite{Langer1984a} provides an exhaustive classification result (see also the excellent lecture notes \cite{Singer2008}).
Here we just collect the facts that we use later for our main theorems.
Recall that a planar curve $\gamma\subset\mathbf{R}^2$ is called {\em wavelike elastica} if there exist $m\in(0,1)$ and $s_0\in\mathbf{R}$ such that, up to dilation, the curve $\gamma$ parameterized by the arclength $s\in[0,L]$ has signed curvature $\mathrm{k}$ of the form $\mathrm{k}(s)=2\sqrt{m}\cn(s-s_0,m)$.

\begin{theorem}[Langer--Singer \cite{Langer1984a,Singer2008}]\label{thm:classificationelastica}
  The following statements hold.
  \begin{enumerate}
    \item Any elastica is contained in an at most three-dimensional affine subspace of $\mathbf{R}^n$.
    %(Therefore, we may only consider spatial elasticae, i.e., an elastica in $\mathbf{R}^3$.)
    \item If an elastica is non-planar, then it has everywhere non-zero curvature (and torsion), where we call an elastica planar (resp.\ non-planar) if it is contained (resp.\ not contained) in a plane.
    \item Any planar elastica with a point of vanishing curvature is either a straight line or a wavelike elastica.
  \end{enumerate}
\end{theorem}

In addition, a {\em figure-eight elastica} is defined to be a wavelike elastica with parameter $m=m^*$.
A key fact is that a figure-eight elastica is characterized by a unique wavelike elastica satisfying a certain Navier boundary condition.

\begin{lemma}\label{lem:Navier}
  Suppose that a wavelike elastica $\gamma:[a,b]\to\mathbf{R}^2$ satisfies the Navier boundary condition that $\gamma(a)=\gamma(b)$ and $\gamma''(a)=\gamma''(b)=0$.
  Then there is a positive integer $N$ such that $\gamma$ is an $\frac{N}{2}$-fold figure-eight elastica (cf.\ Definition \ref{def:figureeight}).
\end{lemma}

This lemma follows if one checks the fact that a figure-eight elastica is the only wavelike elastica that has a self-intersection at their inflection points (where curvature $\mathrm{k}$ changes the sign), recalling that explicit parametrizations of wavelike elasticae are classically known (cf.\ \cite[Chapter XIX, Art.\ 263]{Love1944}).
Here we give a complete argument, where one may refer to \cite{DallAcqua2017,Mueller2021} for precise derivations of the parameterizations via a dynamical system.
Before that, we define the term ``$\frac{N}{2}$-fold'' more precisely (in a general codimension for later use).

\begin{definition}[$\frac{N}{2}$-fold figure-eight elastica]\label{def:figureeight}
  Given a positive integer $N$, we call an immersed curve $\gamma:[a,b]\to\mathbf{R}^n$ {\em $\frac{N}{2}$-fold figure-eight elastica} if $\gamma$ is contained in a plane and its signed curvature $\mathrm{k}(s)$ parameterized by the arclength $s\in[0,L]$ satisfies (up to the choice of the sign) that
  \begin{equation}\label{eq:figureeightcurvature}
    \mathrm{k}(s)=2\sqrt{m^*}\Lambda\cn\left(\Lambda s-K(m^*),m^*\right), \quad \Lambda:=\frac{2K(m^*)N}{L}>0.
  \end{equation}
  Similarly, we call a closed $H^2$-curve $\gamma:\mathbf{T}^1\to\mathbf{R}^n$ {\em $N$-fold closed figure-eight elastica} if there is $t_0\in\mathbf{T}^1$ such that the curve $\tilde{\gamma}:[0,1]\to\mathbf{R}^n$ defined by $\tilde{\gamma}(t):=\gamma(t+t_0)$ is an $N$-fold figure-eight elastica.
\end{definition}

\begin{remark}
  We also call a $\frac{1}{2}$-fold (resp.\ $1$-fold closed) figure-eight elastica {\em half-fold figure-eight elastica} (resp.\ {\em one-fold figure-eight elastica}, or simply {\em figure-eight elastica}).
  Note that in Definition \ref{def:figureeight} the constant $\Lambda$ just plays the role of a scaling factor; namely, if we let the curve $\gamma$ represent the case of $\Lambda=1$, then the general (arclength parameterized) curve $\gamma_\Lambda$ is represented by $\gamma_\Lambda(s)=\frac{1}{\Lambda}\gamma(\Lambda s)$.
\end{remark}

\begin{proof}[Proof of Lemma \ref{lem:Navier}]
  As is derived in \cite[(6.9), (6.10)]{DallAcqua2017}, up to similarity and reparameterization, any wavelike elastica is represented by (a restriction of)
  \begin{equation}\label{eq:wavelike}
    \gamma_\mathrm{wave}^m(s)= \left(
    \begin{array}{c}
      2E(\am(s,m),m)-s\\
      -2\sqrt{m}\cn(s,m) \\
    \end{array}
  \right), \quad \mathrm{k}(s)=2\sqrt{m}\cn(s,m).
  \end{equation}
  Then the change of variables $s=F(x,m)$ yields the simpler representations that
  \begin{equation}\label{eq:wavelike2}
    \tilde{\gamma}_\mathrm{wave}^m(x)= \left(
    \begin{array}{c}
      2E(x,m)-F(x,m)\\
      -2\sqrt{m}\cos{x}\\
    \end{array}
  \right), \quad \tilde{\mathrm{k}}(x)=2\sqrt{m}\cos{x},
  \end{equation}
  so that any zero of $\tilde{\mathrm{k}}(x)$ is of the form $x=p\pi+\frac{\pi}{2}$ with $p\in\mathbf{Z}$.
  Hence, if a wavelike elastica satisfies the Navier boundary condition, then by \eqref{eq:wavelike2} it is necessary that there are some integers $p_1<p_2$ such that $2E(p_1\pi+\frac{\pi}{2},m)-F(p_1\pi+\frac{\pi}{2},m)=2E(p_2\pi+\frac{\pi}{2},m)-F(p_2\pi+\frac{\pi}{2},m)$.
  By periodicity of $E(\cdot,m)$ and $F(\cdot,m)$, we need to have $2E(m)-K(m)=0$, i.e., $m=m^*$.
  Since in addition the Navier boundary condition implies that the curvature has to be zero at the endpoints, we obtain representation \eqref{eq:figureeightcurvature}.
\end{proof}

We summarize a few more basic properties of $\frac{N}{2}$-fold figure-eight elasticae that we use in this paper, cf.\ Figure \ref{fig:figureeight}.

\begin{figure}[htbp]
  \includegraphics[width=40mm]{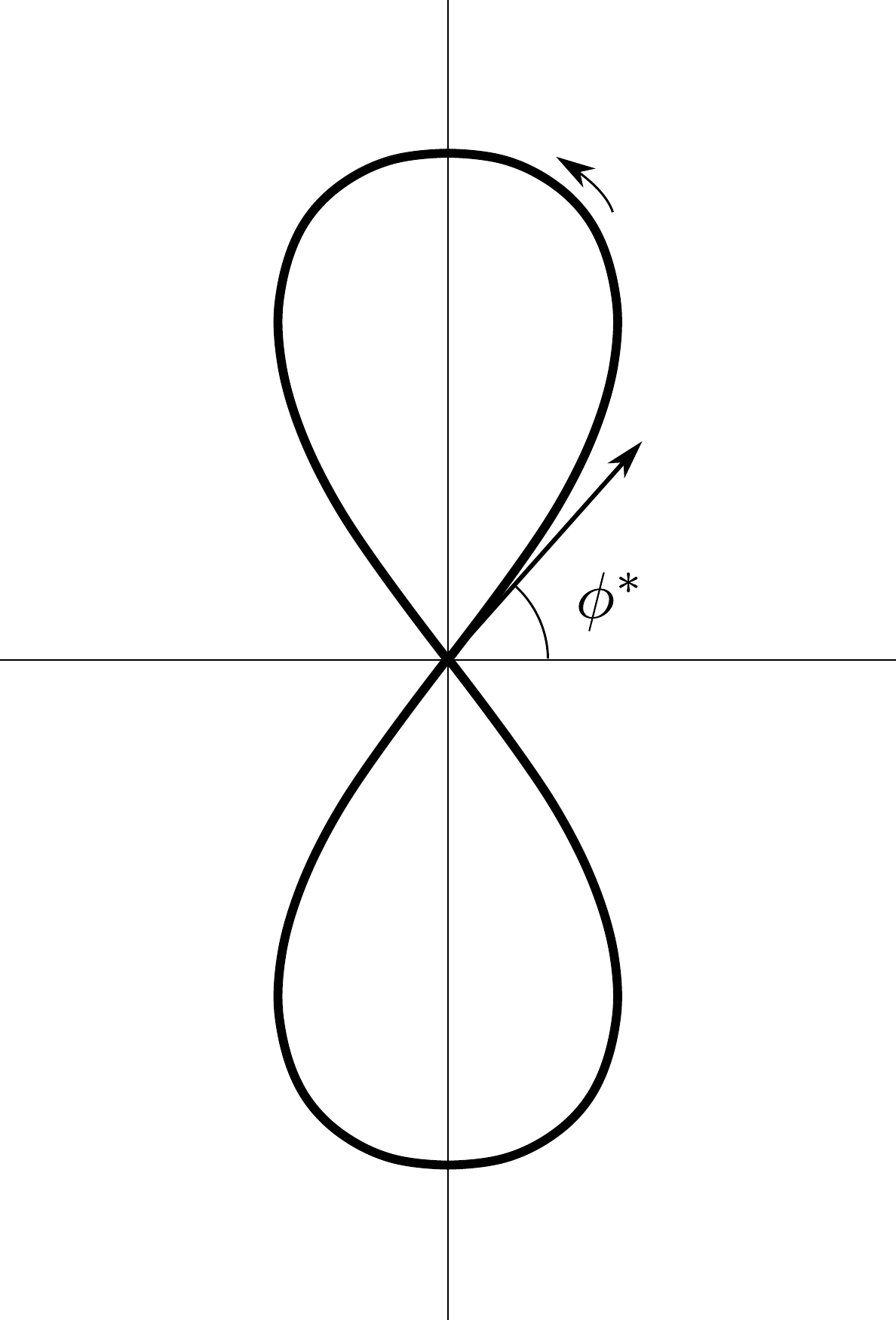}
  \caption{Figure-eight elastica.}
  \label{fig:figureeight}
\end{figure}

\begin{lemma}[Basic properties of figure-eight elasticae]\label{lem:figureeight}
  Let $\gamma$ be an $\frac{N}{2}$-fold figure-eight elastica in $\mathbf{R}^n$.
  Then, up to reparameterization and up to similarities, $\gamma$ is contained in the plane $\mathbf{R}^2\simeq\mathbf{R}^2\times\{0\}\subset\mathbf{R}^n$ and its arclength parameterization $\gamma:[0,2NK(m^*)]\to\mathbf{R}^2$ is given by $\gamma(s)= -\gamma^{m^*}_\mathrm{wave}(s-K(m^*))$, that is,
  \begin{equation}\label{eq:figureeightcurve}
    \gamma(s)= \left(
    \begin{array}{c}
      -2E(\am( s-K(m^*),m^*),m^*) + ( s -K(m^*))\\
      2\sqrt{m^*}\cn(s-K(m^*),m^*) \\
    \end{array}
  \right).
  \end{equation}
  In addition, the above representation satisfies the following properties.
  \begin{enumerate}
    \item The curve $\gamma$ passes through the origin if and only if $s=0,2K(m^*),\dots,2NK(m^*)$.
    In addition, these points also characterize those where the signed curvature $\mathrm{k}(s)=2\sqrt{m^*}\cn(s-K(m^*),m^*)$ vanishes.
    \item The curve $\gamma$ possesses the symmetry that
    $$\gamma(2K(m^*)-s)=P_1\gamma(s) \quad \text{for}\ s\in[0,K(m^*)],$$
    and the periodicity that
    $$\gamma(s+2K(m^*))=P_2\gamma(s) \quad \text{for}\ s\in[0,2(N-1)K(m^*)],$$
    where $P_i$ denotes the reflection with respect to the $i$-th component $p\mapsto p-2\langle p,e_i \rangle e_i$.
    \item The normalized bending energy is given by
    $$\bar{B}[\gamma]=\varpi^*N^2.$$
    In particular, $\varpi^*$ and $4\varpi^*$ are the normalized bending energy of a half-fold and one-fold figure-eight elastica, respectively.
    \item Let $\phi^*\in(0,\pi/2)$ be the unique angle such that $\cos\phi^*=2m^*-1>0$, cf.\ \eqref{eq:parameter0.5}.
    Then
    $$\langle\gamma'(0),\gamma'(2K(m^*))\rangle = \cos2\phi^*.$$
  \end{enumerate}
\end{lemma}

\begin{proof}
  Curve representation \eqref{eq:figureeightcurve} directly follows by Definition \ref{def:figureeight} with rescaling $\Lambda=1$ and by representation \eqref{eq:wavelike} with parameter-shifting $\gamma(s)\mapsto \gamma(s-K(m^*))$ and $\pi$-angle rotation $\gamma\mapsto -\gamma$.
  (Our representation is chosen so that $\gamma(0)$ is the origin and $\gamma'(0)$ is contained in the first quadrant, i.e., $(\gamma^1)'(0)>0$ and $(\gamma^2)'(0)>0$, where $\gamma^i$ denotes the $i$-th component; see below for the proof.)

  Property (i) follows by definition of $m^*$ and by periodicity of elliptic functions and integrals.
  Indeed, thanks to the fact that $\cn(-K(m),m)=0$ and the particular property that $-2E(\am(-K(m^*),m^*),m^*)-K(m^*)= 2E(m^*)-K(m^*)=0$, the curve $\gamma$ passes through the origin at $s=0$ and hence, by periodicity, also at $s=2K(m^*),\dots,2NK(m^*)$.
  Those points also characterize the zeroes of curvature, cf.\ \eqref{eq:figureeightcurvature}.

  Property (ii) directly follows by symmetry and periodicity of the functions involved in representation \eqref{eq:figureeightcurve}.

  We prove property (iii).
  By periodicity we may only compute the case of $N=1$.
  Note that $L[\gamma] = 2K(m^*)$.
  In addition, by representation \eqref{eq:figureeightcurvature} with $N=1$, and by even-symmetry of the function $\cn(\cdot,m^*)$, we compute
  \begin{align*}
    B[\gamma] &= \int_{0}^{2K(m^*)}\big( 2\sqrt{m^*}\cn(s-K(m^*),m^*) \big)^2ds\\
    &= 8m^*\int_{0}^{K(m^*)}\cn^2(s,m^*)ds \\
    &= 8m^* \int_{0}^{\pi/2} \frac{\cos^2{x}}{\sqrt{1-m^*\sin^2{x}}} dx \quad (s=F(x,m^*))\\
    &= 8\big(E(m^*)-(1-m^*)K(m^*)\big).
  \end{align*}
  Since $2E(m^*)=K(m^*)$, we see that $\bar{B}[\gamma]=L[\gamma]B[\gamma]=\varpi^*$, cf.\ \eqref{eq:constant}.

  Before proving property (iv) we first compute the tangent vector at the origin, in particular confirming that $(\gamma^1)'(0)>0$ and $(\gamma^2)'(0)>0$.
  By a similar change of variables as in \eqref{eq:wavelike2}, we deduce that, for a unique $c>0$ (making $|\gamma'(0)|=1$),
  $$(\gamma^1)'(0)=c\tfrac{\partial}{\partial x}(-2E(x,m^*)+F(x,m^*))|_{x=-\frac{\pi}{2}}, \ (\gamma^2)'(0)=c\tfrac{\partial}{\partial x}(2\sqrt{m^*}\cos{x})|_{x=-\frac{\pi}{2}}.$$ 
  Direct computations imply that $(\gamma^1)'(0)=c\frac{2m^*-1}{\sqrt{1-m^*}}$ and $(\gamma^2)'(0)=c2\sqrt{m^*}$.
  By taking the explicit $c=\sqrt{1-m^*}$ we deduce that
  $$(\gamma^1)'(0)=2m^*-1>0, \quad (\gamma^2)'(0)=2\sqrt{m^*}\sqrt{1-m^*}>0,$$
  where the former positivity relies on the fact that $m^*>0.5$, cf. \eqref{eq:parameter0.5}.
  Now we prove (iv).
  By periodicity in (ii), it is now sufficient to prove that the tangent vector $\gamma'(0)$ and the vector $e_1=(1,0)^\top$ make the angle $\phi^*$, or equivalently, that $(\gamma^1)'(0)=\cos\phi^*$.
  The above computation of $(\gamma^1)'(0)=2m^*-1$ and the definition that $\cos\phi^*=2m^*-1$ complete the proof.
  (See also \cite{Djondjorov2008} for a different derivation of the angle.)
\end{proof}

We now state a key proposition for our main theorems, which is formulated in terms of a minimizing problem of the bending energy subject to a zeroth-order boundary condition.
Hereafter we mainly deal with open curves with $H^2$ ($=W^{2,2}$) Sobolev regularity, where this regularity is natural in view of the bending energy.
Note that by Sobolev embedding $H^2\hookrightarrow C^1$ such curves still possess pointwise meaning up to first order; in particular, both immersedness and multiplicity are well defined.

\begin{proposition}[Minimality of half-fold figure-eight elasticae]\label{prop:minimizer}
  Let $\gamma:[a,b]\to\mathbf{R}^n$ be an immersed $H^2$-curve such that $\gamma(a)=\gamma(b)$.
  Then
  $$\bar{B}[\gamma]\geq\varpi^*,$$
  where equality is attained if and only if $\gamma$ is a half-fold figure-eight elastica.
\end{proposition}

As is already emphasized in the introduction, the choice of this zeroth-order boundary condition is a key idea in our strategy.
Since Proposition \ref{prop:minimizer} is later used for each part of the original curve after division at a multiplicity point, it seems natural to choose an up-to-first-order (clamped) boundary condition in view of $H^2\hookrightarrow C^1$.
However for such a condition the minimum value sensitively depends on a first-order quantity so that there remains an additional issue of complicated energy competition.
Instead, here we first relax the boundary condition to deduce a geometrically unique minimizer, and then (in Section \ref{sect:rigidity}) consider whether a collection of such minimizers can be applied to the original problem.

For convenience we introduce a class of unit-speed curves.
Let $I:=(0,1)$ and $\bar{I}:=[0,1]$.
Let $X$ be the class of all unit-speed curves $\gamma\in H^2(I;\mathbf{R}^n)\hookrightarrow C^1(\bar{I};\mathbf{R}^n)$, that is,
$$X:=\{\gamma\in H^2(I;\mathbf{R}^n) \mid |\gamma'|\equiv1 \}.$$
Note that for any $\gamma\in X$ we have $L[\gamma]=1$ and thus $\bar{B}[\gamma]=B[\gamma]$.
This setting of arclength parameterization does not lose generality thanks to the invariance of our problem up to rescaling and reparameterization.
Finally, let
$$X_0:=\{\gamma\in X \mid \gamma(0)=\gamma(1)=0\}.$$

\begin{proof}[Proof of Proposition \ref{prop:minimizer}]
  Up to similarity and reparameterization, we may only argue within the class $X_0$.
  Thus, it is sufficient to consider the (unnormalized) bending energy $B$ and prove the following properties: There exists $\bar{\gamma}\in X_0$ such that $B[\bar\gamma]=\inf_{X_0}B$, such a minimizer $\bar{\gamma}$ must be a half-fold figure-eight elastica, and $B[\bar{\gamma}]=\varpi^*$.

  The existence of a minimizer follows from the standard direct method, which we demonstrate here for the reader's convenience (and for using a similar argument later).
  Let $\{\gamma_j\}\subset X_0$ be a minimizing sequence $B[\gamma_j]\to\inf_{X_0}B$.
  Then, combining this limit with the fact that $\gamma_j(0)=0$ ($=\gamma_j(1)$) and $|\gamma_j'|\equiv1$, we find that $\{\gamma_j\}$ is bounded in $H^2(I;\mathbf{R}^n)$ so that there is a subsequence (without relabeling) that converges in the senses of $H^2$-weak and $C^1$.
  The limit curve $\bar\gamma$ is thus a unit-speed curve in $H^2(I;\mathbf{R}^n)$ such that $\bar\gamma(0)=\bar\gamma(1)=0$, i.e., $\bar\gamma\in X_0$, and the weak lower semicontinuity ensures that
  $$\inf_{X_0}B=\liminf_{j\to\infty}B[\gamma_j]\geq B[\bar\gamma].$$
  This means that $\bar\gamma$ is a minimizer, completing the proof of existence.

  Now we prove that any minimizer must be a half-fold figure-eight elastica.
  Fix any minimizer $\gamma\in X_0$ (note that $|\gamma'|\equiv1$).
  Thanks to the reparameterization invariance, the curve $\gamma$ is also a minimizer in the class of (not necessary unit-speed but only) immersed curves with the same boundary condition.
  Then the standard Lagrange multiplier method \cite[Proposition 1, Sect.\ 4.14]{Zeidler1995} ensures that there exists $\lambda\in\mathbf{R}$ such that $\gamma$ is a critical point of the functional $E_\lambda$, cf.\ \eqref{eq:energy_E_lambda}.
  More precisely, for any $\eta\in H^2(I;\mathbf{R}^n)\cap H^1_0(I;\mathbf{R}^n)$ we have the $H^2$-continuous Fr\'{e}chet derivatives given by
  \begin{align*}
    DL[\gamma](\eta) = \int_I \langle\gamma',\eta'\rangle dt, \quad DB[\gamma](\eta) = \int_I 2\langle \gamma'',\eta'' \rangle dt - \int_I 3|\gamma''|^2\langle \gamma',\eta' \rangle dt,
  \end{align*}
  where we used $|\gamma'|\equiv1$ to reduce the formulae; then the minimality of $\gamma$ implies that $DE_\lambda[\gamma]=DB[\gamma]+\lambda DL[\gamma]=0$ for some $\lambda\in\mathbf{R}$, where we used the fact that $DL[\gamma]$ is a non-zero functional since $\gamma$ is not a segment.
  Since this identity holds in particular for all $\eta\in C_c^\infty(I;\mathbf{R}^n)$, we first deduce from a standard bootstrap argument that $\gamma\in C^\infty(\bar{I};\mathbf{R}^n)$.
  Then we deduce from integration by parts (see also \cite[Lemma A.1]{DallAcqua2014}) that for all $\eta \in C^\infty(\bar{I};\mathbf{R}^n)$ such that $\eta(0)=\eta(1)=0$,
  $$DE_\lambda[\gamma](\eta)= \int_I \langle 2\nabla_s^2\kappa+|\kappa^2|\kappa -\lambda\kappa,\eta \rangle + [2\langle \kappa,\nabla_s\eta \rangle]_0^1 =0.$$
  This implies that $\gamma$ is an elastica, cf.\ \eqref{eq:elastica}, and in addition, $\kappa(0)=\kappa(1)=0$.
  Combining these facts with the original boundary condition that $\gamma(0)=\gamma(1)=0$, and using Theorem \ref{thm:classificationelastica} and Lemma \ref{lem:Navier}, we find that the minimizer $\gamma$ must be an $\frac{N}{2}$-fold figure-eight elastica for some positive integer $N$.
  By Lemma \ref{lem:figureeight} (iii) and the energy-minimality of $\gamma$, we have $N=1$.
  We thus conclude that any minimizer is a half-fold figure-eight elastica.

  Finally, by Lemma \ref{lem:figureeight} (iii), we deduce that the minimum is $\varpi^*$.
\end{proof}

\subsection{Li--Yau type multiplicity inequality}\label{subsect:Li-Yau}

We now turn to the proof of Theorem \ref{thm:Li-Yau_closed}.
For later use it is convenient to first prove an open-curve counterpart of Theorem \ref{thm:Li-Yau_closed}.

\begin{theorem}[Multiplicity inequality for open curves]\label{thm:Li-Yau_open}
  Let $\gamma:[0,1]\to\mathbf{R}^n$ be an immersed curve with a point of multiplicity $k\geq2$.
  Then
  \begin{equation}\label{eq:Li-Yau_open}
    \bar{B}[\gamma]\geq \varpi^*(k-1)^2.
  \end{equation}
  In particular, if an immersed curve $\gamma$ has the property that $\bar{B}[\gamma]<\varpi^*$, then $\gamma$ is embedded; the threshold $\varpi^*$ is optimal due to a half-fold figure-eight elastica.
\end{theorem}

\begin{proof}[Proof of Theorem \ref{thm:Li-Yau_open}]
  By the assumption on multiplicity, there are $0\leq a_1<\dots<a_k\leq1$ such that $\gamma(a_1)=\dots=\gamma(a_k)$.
  If $a_1>0$, then we may cut off the part $\gamma|_{[0,a_1]}$ and create a new curve whose normalized bending energy is strictly less than the original one, and hence without loss of generality we may assume that $a_1=0$.
  Similarly, we may assume that $a_k=1$.
  Up to translation we may assume that the multiplicity point is the origin, and hence can apply Proposition \ref{prop:minimizer} to each $\gamma_i:=\gamma_{[a_i,a_{i+1}]}$, where $i=1,\dots,k-1$, to deduce that
  \begin{equation}\label{eq01}
    L[\gamma_i]B[\gamma_i] = \bar{B}[\gamma_i] \geq \varpi^*.
  \end{equation}
  Noting that $B[\gamma]=\sum_{i=1}^{k-1}B[\gamma_i]$ and $L[\gamma]=\sum_{i=1}^{k-1}L[\gamma_i]$, we have
  \begin{equation}\label{eq02}
    \begin{split}
        \bar{B}[\gamma] &= \Big(\sum_{i=1}^{k-1}L[\gamma_i]\Big)\Big(\sum_{i=1}^{k-1}B[\gamma_i]\Big) \\
        &\geq \Big(\sum_{i=1}^{k-1}L[\gamma_i]\Big)\Big(\sum_{i=1}^{k-1}\frac{1}{L[\gamma_i]}\varpi^*\Big)\geq (k-1)^2\varpi^*,
    \end{split}
  \end{equation}
  where the last estimate follows by the elementary inequality of arithmetic and harmonic means.
\end{proof}

Theorem \ref{thm:Li-Yau_closed} can be regarded as a special consequence of Theorem \ref{thm:Li-Yau_open}.

\begin{proof}[Proof of Theorem \ref{thm:Li-Yau_closed}]
  Given any closed curve $\gamma$ with a point of multiplicity $k$, we can create an open curve with a point of multiplicity $k+1$ after cutting $\gamma$ at the original point of multiplicity and opening the domain $\mathbf{T}^1$ to $[0,1]$.
  Applying Theorem \ref{thm:Li-Yau_open} to this curve, we obtain the improved inequality in Theorem \ref{thm:Li-Yau_closed}.
\end{proof}

\section{Leafed elastica and optimality}\label{sect:rigidity}

The goal of this section is to prove Theorems \ref{thm:rigidity_closed} and \ref{thm:nonoptimality}.
To this end we introduce the notion of {\em leafed elastica}, which is compatible with our problem.
We first discuss some basic properties of leafed elasticae, from which we observe how the difference depending on the pair $(n,k)$ occurs in optimality.

\subsection{Leafed elastica}\label{subsect:leafedelastica}

Leafed elasticae are defined by connecting leaves of half figure-eight elasticae, cf.\ Figure \ref{fig:leaf}.

\begin{figure}[htbp]
  \includegraphics[width=18mm]{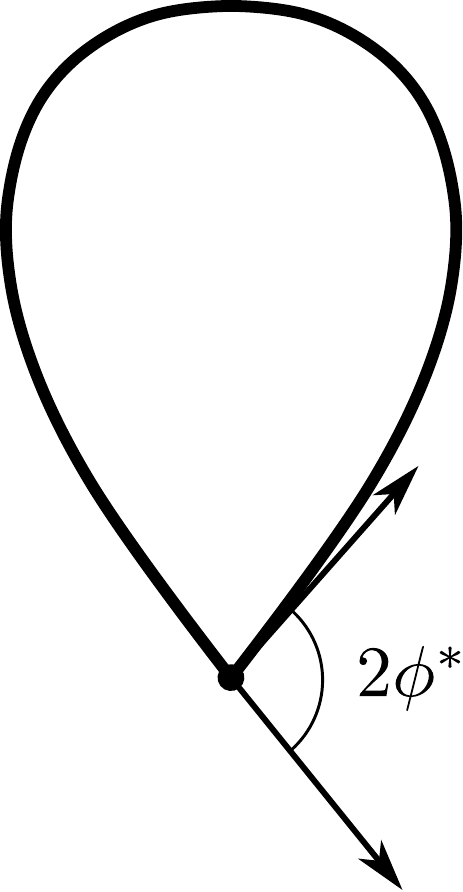}
  \caption{Leaf: A half-fold figure-eight elastica.}
  \label{fig:leaf}
\end{figure}

\begin{definition}[Leafed elastica]\label{def:leafedelastica}
  Let $n\geq2$ and $k\geq1$.
  We call an immersed $H^2$-curve $\gamma:[a,b]\to\mathbf{R}^n$ {\em $k$-leafed elastica} if there are $a=a_0<a_1<\dots<a_k=b$ such that for each $i=1,\dots,k$ the curve $\gamma_i:=\gamma|_{[a_{i-1},a_i]}$ is a half-fold figure-eight elastica, and also $L[\gamma_1]=\dots=L[\gamma_k]$.
  Similarly, we call a closed $H^2$-curve $\gamma:\mathbf{T}^1\to\mathbf{R}^n$ {\em closed $k$-leafed elastica} if there is $t_0\in\mathbf{T}^1$ such that the curve $\tilde{\gamma}:[0,1]\to\mathbf{R}^n$ defined by $\tilde{\gamma}(t):=\gamma(t+t_0)$ is a $k$-leafed elastica.
  In addition, to specify the dimension $n$ of the target space, sometimes we also use the term {\em (closed) $(n,k)$-leafed elastica}.
\end{definition}

\begin{remark}\label{rem:joint}
  By definition, a $k$-leafed elastica $\gamma$ has a point of multiplicity $k+1$, namely $\gamma(a_0)=\dots=\gamma(a_k)$.
  Also, a closed $k$-leafed elastica has a point of multiplicity $k$.
  We call such a point {\em joint} of a leafed elastica.
\end{remark}

An easy consequence of the previous definition is the following:

\begin{proposition}[Regularity and energy]\label{prop:regularityenergy_leafedelastica}
  Let $k\geq1$.
  Then any $k$-leafed (resp.\ closed $k$-leafed) elastica is of class $C^{2,1}$ and piecewise analytic, has a point of multiplicity $k+1$ (resp.\ $k$), and has the energy $\bar{B}[\gamma]=\varpi^*k^2$.
\end{proposition}

\begin{proof}
  The piecewise analyticity follows by the representation in \eqref{eq:figureeightcurve}.
  The whole $C^{2,1}$-regularity follows in this way; first, any $k$-leafed elastica $\gamma$ is automatically of class $C^1$ since $H^2\hookrightarrow C^1$; also, the curvature of $\gamma$ vanishes at each joint of leaves by Lemma \ref{lem:figureeight} (i) so that $\gamma$ is also of class $C^2$; finally, the third derivative of $\gamma$ is bounded in $L^\infty$ since each (planar) leaf of $\gamma$ has the signed curvature given in \eqref{eq:figureeightcurvature} (in an affine subspace $\simeq\mathbf{R}^2$) so that, by the derivative formula
  \begin{equation}\label{eq:derivativecn}
    \tfrac{\partial}{\partial x}\cn(x,m)=-\sn(x,m)\sqrt{1-m\sn^2(x,m)},
  \end{equation}
  its derivative is bounded; consequently, $\gamma$ is of class $W^{3,\infty}=C^{2,1}$.
  The multiplicity at the joint is already discussed in Remark \ref{rem:joint}.
  Finally, the normalized bending energy can be explicitly computed, cf.\ Lemma \ref{lem:figureeight} (iii).
\end{proof}

We mention two obvious examples of classical figure-eight elasticae.

\begin{example}\label{ex:openleafedfigureeight}
  Let $k\geq1$.
  Then a $\frac{k}{2}$-fold figure-eight elastica is an analytic planar example of an $(n,k)$-leafed elastica.
\end{example}

\begin{example}\label{ex:closedevenleafedfigurefight}
  Let $k\geq2$ be even.
  Then a closed $\frac{k}{2}$-fold figure-eight elastica is an analytic planar example of a closed $(n,k)$-leafed elastica.
\end{example}

Leafed elasticae have flexibility due to possible discontinuity of their third derivatives at the joints, and in particular these examples do not exhaust all possible configurations.
For example, it is easy to imagine that for any $k\geq2$ there are more (open) $k$-leafed elasticae by reflecting or twisting leaves at their joints arbitrarily.

On the other hand, to obtain closed $k$-leafed elasticae, we are required to close it up in the first-order sense.
This requirement causes non-negligible rigidity.
Indeed, it is easy to observe

\begin{proposition}\label{prop:uniqunessleaf12}
  No closed $1$-leafed elastica exists.
  A closed $2$-leafed elastica is (up to similarity and reparameterization) uniquely given by a figure-eight elastica.
\end{proposition}

\begin{proof}
  This follows since the angle made by the tangent vectors at the endpoints of one leaf is given by $2\phi^*\in(0,\pi)$, cf.\ Lemma \ref{lem:figureeight} and Figure \ref{fig:leaf}.
\end{proof}

In general, whether there exists a closed $(n,k)$-leafed elastica can be characterized by whether the endpoints of leaves can be joined up to first order.
This fact is summarized in the following lemma, the proof of which is straightforward and safely omitted.

\begin{lemma}[Characterization of closed $(n,k)$-leafed elasticae]
  \label{lem:closedleafedcharacterization}
  Let $n\geq2$ and $k\geq1$.
  Let $\Omega^*(n,k)$ be the set of all $k$-tuples $(\omega_1,\dots,\omega_k)$ of $n$-dimensional unit-vectors $\omega_1,\dots,\omega_k\in \mathbf{S}^{n-1}\subset\mathbf{R}^n$ such that $\langle\omega_{i},\omega_{i-1}\rangle=\cos2\phi^*$ holds for any $i=1,\dots,k$, where we interpret $\omega_0:=\omega_k$.
  %where $\langle\cdot,\cdot\rangle$ denotes the Euclidean inner product in $\mathbf{R}^n$.

  \begin{itemize}
      \item[(i)] If $\gamma:\mathbf{T}^1\to\mathbf{R}^n$ is a unit-speed closed $(n,k)$-leafed elastica, then there is $t_0\in\mathbf{T}^1$ such that the $k$-tuple of vectors $\omega_i:=\gamma'(\frac{i}{k}+t_0)$, $i=1,\dots,k$, is an element of $\Omega^*(n,k)$.
      \item[(ii)] Conversely, %for any element $(\omega_1,\dots,\omega_k)\in\Omega^*(n,k)$
      if $(\omega_1,\dots,\omega_k)\in (\mathbf{S}^{n-1})^k$ is an element of $\Omega^*(n,k)$, then there exists a unit-speed closed $(n,k)$-leafed elastica $\gamma:\mathbf{T}^1\to\mathbf{R}^n$ such that $\gamma(0)=0$ and $\gamma'(\frac{i}{k})=\omega_i$ for $i=1,\dots,k$.
  \end{itemize}
\end{lemma}

%This is in particular consistent with Example \ref{ex:closedevenleafedfigurefight}, namely the fact that for every $n\geq2$ and every even $k\geq2$ there exists a closed $(n,k)$-leafed elastica, since after choosing any $\omega_0$ and $\omega_1$ such that $\langle\omega_0,\omega_1\rangle=\cos2\phi^*$ we may just define $\omega_{i+2}:=\omega_i$ for $i=0,\dots,k-1$.

This characterization is useful both for ensuring nonexistence and for constructing concrete examples.

We first verify nonexistence of planar closed leafed elasticae with odd leaves caused by an algebraic obstruction for closing the leaves (there always exists a planar closed $k$-leafed elastica if $k$ is even, cf.\ Example \ref{ex:closedevenleafedfigurefight}).
This obstruction is related to an irrationality or transcendence problem involving special values of hypergeometric functions.
Such a problem is in general difficult in spite of its classicality, see e.g.\ \cite{Fischler2020} and references therein.
However, fortunately, our problem can be reduced to the following theorem of Andr\'{e} \cite{Andre1996}.
Let $\overline{\mathbf{Q}}\subset\mathbf{C}$ denote the field of algebraic numbers.

\begin{theorem}[Andr\'{e} \cite{Andre1996}]
  For any $z\in\overline{\mathbf{Q}}$ with $0<|z|<1$, the values of the Gaussian hypergeometric functions $x_1:={}_2F_1[\frac{1}{2},\frac{1}{2};1;z]$ and $x_2:={}_2F_1[-\frac{1}{2},\frac{1}{2};1;z]$
  are algebraically independent over $\overline{\mathbf{Q}}$, i.e., those two values do not satisfy $f(x_1,x_2)=0$ for any non-trivial $f\in \overline{\mathbf{Q}}[X_1,X_2]$ (polynomial with coefficients in $\overline{\mathbf{Q}}$).
\end{theorem}

With the aid of this result we prove the following

\begin{proposition}
  \label{prop:nonexistenceleaf}
  For any odd $k\geq3$ there exists no closed $(2,k)$-leafed elastica.
\end{proposition}

\begin{proof}
  We first note that for any given $\omega\in\mathbf{S}^1$ there are only two possibilities of $\omega'\in\mathbf{S}^1$ such that $\langle \omega,\omega' \rangle=\cos2\phi^*$, namely $\omega'=R_\pm\omega$, where $R_\pm\in\mathrm{SO}(2)$ denotes the (counterclockwise) rotation matrix through angle $\pm2\phi^*$.
  Combining this fact with Lemma \ref{lem:closedleafedcharacterization}, we deduce that the assertion is equivalent to the following statement:
  For any odd integer $k\geq3$ there is no $k$-tuple of rotation matrices $R_1,\cdots,R_k\in\mathrm{SO}(2)$ through angle either $2\phi^*$ or $-2\phi^*$ such that $R_k\cdots R_1=I$, where $I$ denotes the identity matrix; in other words, for any odd $k\geq3$ there exists no $k$-tuple
  $(\sigma_1,\dots,\sigma_k)\subset\{-1,1\}^k$ such that $\sum_{i=1}^k\sigma_i2\phi^*\in 2\pi\mathbf{Z}$.

  Now we prove that the above equivalent statement holds true.
  It is sufficient to prove that $\phi^*\not\in\pi\mathbf{Q}$.
  By the well-known relation (e.g.\ found in \cite[17.3.9, 17.3.10]{Abramowitz1992}) that ${}_2F_1[\frac{1}{2},\frac{1}{2};1;m]=\frac{2}{\pi}K(m)$ and ${}_2F_1[-\frac{1}{2},\frac{1}{2};1;m]=\frac{2}{\pi}E(m)$ for $m\in(0,1)$, and by definition of $m^*$,
  we deduce that
  $${}_2F_1[\tfrac{1}{2},\tfrac{1}{2};1;m^*]-2{}_2F_1[-\tfrac{1}{2},\tfrac{1}{2};1;m^*]=0.$$
  By Andr\'{e}'s theorem, the number $m^*\in(0,1)$ is transcendental, and hence so is $\cos\phi^*$ ($=2m^*-1$).
  This implies the desired irrationality; indeed, for any rational angle $\phi=(p/q)\pi$ with $p/q\in\mathbf{Q}$ we have $T_q(\cos\phi)=\cos(q\phi)\in\{\pm1\}$, where $T_q$ denotes the $q$-th Chebyshev polynomial of the first kind, and hence $\cos\phi\in\overline{\mathbf{Q}}$.
\end{proof}

In contrast, if the codimension is higher, namely $n\geq3$, then we can still construct a closed $(n,k)$-leafed elastica for any multiplicity $k\geq2$.
We may again focus on the case of odd $k\geq3$ thanks to Example \ref{ex:closedevenleafedfigurefight}.
To construct a key example we give an analytic estimate of the angle $\phi^*\in(0,\pi/2)$ ($\approx 49.290^\circ$) in the form of $45^\circ<\phi^*<60^\circ$.
The upper bound plays a crucial role in Example \ref{ex:propeller}, while the lower bound is used in Remark \ref{rem:nonuniqueness} and also plays a key role in the coming application to networks, namely in the proof of Lemma \ref{lem:degenerate_network} below.
To this end we verify and use the following fact:

\begin{lemma}\label{lem:parameter}
  %There exists a unique parameter $m^*\in(0,1)$ such that $2E(m^*)=K(m^*)$. In addition,
  $0.75<m^*<0.85$.
\end{lemma}

The proof is postponed to Appendix \ref{sec:app:figure-eight}.
This lemma immediately implies

\begin{lemma}\label{lem:angle}
  $\pi/4<\phi^*<\pi/3$.
\end{lemma}

\begin{proof}
  By definition $\cos\phi^*=2m^*-1$ of $\phi^*$, %and monotonicity of $\phi\mapsto\tan{\phi}$ on $(0,\pi/2)$,
  it is sufficient to prove that $\cos(\pi/3)=1/2<2m^*-1<1/\sqrt{2}=\cos(\pi/4)$.
  Thanks to Lemma \ref{lem:parameter}, this can be reduced to %confirming that $[1+\frac{1}{5-2\sqrt{5}}]^{-\frac{1}{2}}<2\cdot\frac{8}{10}-1$ and that $2\cdot\frac{85}{100}-1<2^{-\frac{1}{2}}$. These are reduced to
  easy comparisons of rational numbers.
\end{proof}

The key example can be now constructed as follows:

\begin{example}[Elastic propeller]\label{ex:propeller}
  Let $\omega_1,\omega_2,\omega_3\in\mathbf{S}^2\subset\mathbf{R}^3$ be taken so that the triple of them is an element of $\Omega^*(3,3)$, cf.\ Lemma \ref{lem:closedleafedcharacterization}.
  In fact, such a triple exists and is unique up to rigid motions, by the fact that $2\phi^*\in(0,2\pi/3)$, cf.\ Lemma \ref{lem:angle}, and by the elementary geometry that the vectors $\omega_1,\omega_2,\omega_3$ need to make a triangular pyramid whose one face is an equilateral triangle and the others are all congruent to the same isosceles triangle of angle $2\phi^*$, cf.\ Figure \ref{fig:propeller} (left).
  By using such a triple, in view of Lemma \ref{lem:closedleafedcharacterization}, we can construct %as in Lemma \ref{lem:closedleafedcharacterization}
  an example of a closed $(3,3)$-leafed elastica for $n\geq3$, cf.\ Figure \ref{fig:propeller} (right), which is thus unique up to similarity and reparameterization.
  We call it {\em elastic propeller}.
  %For completeness we check the elementary unique existence of the above triple.
  %Deduce by an elementary geometry that such a triple must be represented by $\omega_i=A_i$, where $O$ denotes the origin and $A_1,A_2,A_3$ lie in $\mathbf{S}^2$ and make an equilateral triangle so that $OA_1A_2A_3$ makes a triangular pyramid in $\mathbf{R}^3$.
  %In such a class of triangular pyramids, the configuration is (up to congruent) uniquely characterized by the angle $\theta\in[0,2\pi/3]$ made by (any) two of the vectors departing from the origin.
  %Since $2\phi^*\in(0,2\pi/3)$ by Lemma \ref{lem:angle}, we may find a unique configuration such that $\theta=2\phi^*$, completing the proof.
\end{example}

\begin{figure}[htbp]
  \includegraphics[width=100mm]{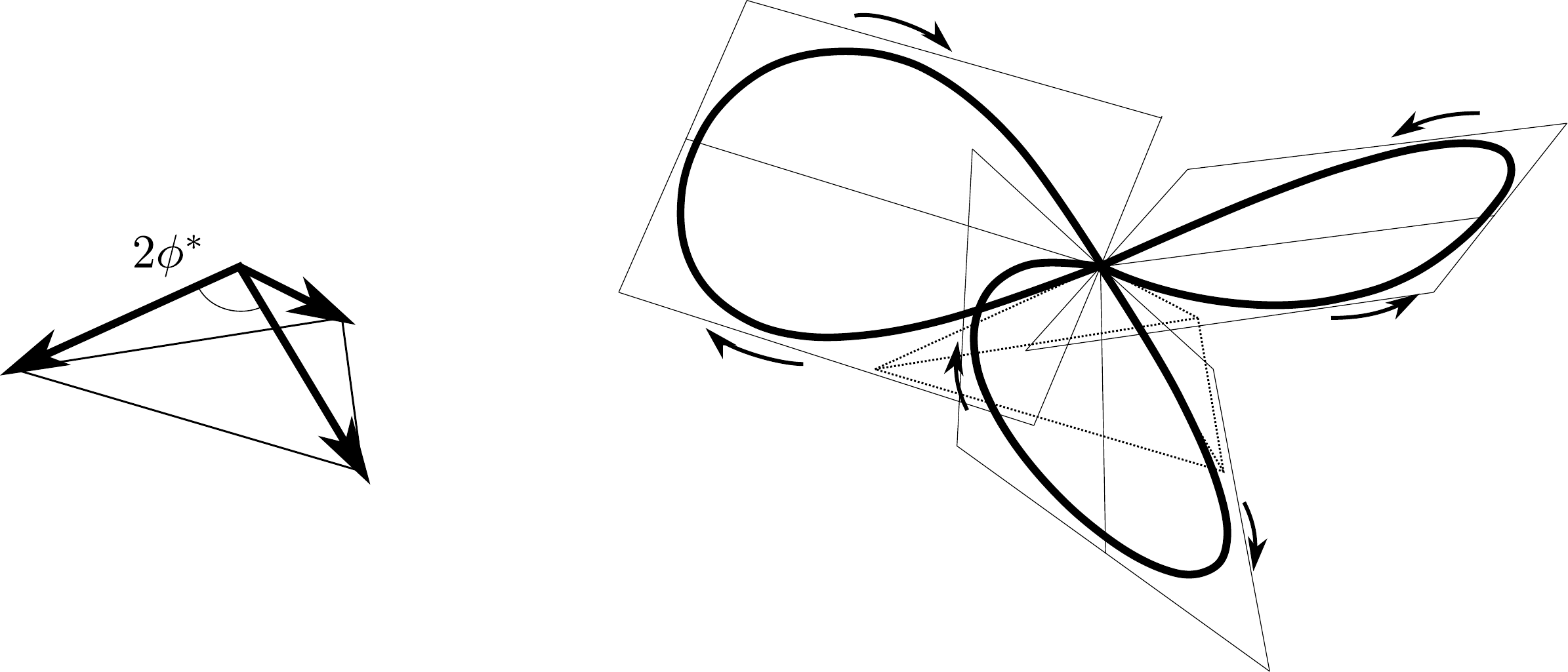}
  \caption{Elastic propeller: A unique $3$-leafed elastica ($n\geq3$).}
  \label{fig:propeller}
\end{figure}

The elastic propeller is the only example of a closed $3$-leafed elastica.

\begin{proposition}\label{prop:uniquenessleaf3}
  Let $n\geq3$.
  Then a closed $(n,3)$-leafed elastica is (up to similarity and reparameterization) uniquely given by an elastic propeller.
\end{proposition}

\begin{proof}
  This follows by the uniqueness of a closed $(3,3)$-leafed elastica and the fact that any closed $(n,3)$-leafed elastica must be contained in a three-dimensional affine subspace spanned by the three leaves.
\end{proof}

An elastic propeller can be also used to construct (non-symmetric) closed leafed elasticae with any higher odd number of leaves in a unified manner.

\begin{example}[Closed $(n,k)$-leafed elastica for $n\geq3$ and odd $k\geq5$]\label{ex:propellereight}
  For any $n\geq3$ and any odd $k\geq5$ we can construct an example of a (three-dimensional) closed $(n,k)$-leafed elastica by connecting an elastic propeller and a $\frac{k-3}{2}$-fold closed figure-eight elastica.
\end{example}

In summary, we obtain

\begin{proposition}\label{prop:existence_closedleafedelastica}
  Let $n\geq2$ and $k\geq1$.
  Either if $n\geq3$ or if $k$ is even, then there exists a closed $(n,k)$-leafed elastica.
\end{proposition}

\begin{proof}
  It follows from Example \ref{ex:closedevenleafedfigurefight}, Example \ref{ex:propeller}, and Example \ref{ex:propellereight}.
\end{proof}

\subsection{Optimality and rigidity in the multiplicity inequality}\label{subsect:rigidity}

From now on we prove Theorems \ref{thm:rigidity_closed} and \ref{thm:nonoptimality} by using the above results on leafed elasticae.
We first state and prove an open-curve counterpart of Theorem \ref{thm:rigidity_closed}, which holds in full generality.

\begin{theorem}[Optimality and rigidity for open curves]\label{thm:rigidity_open}
  Let $n\geq2$ and $k\geq2$.
  Then there exists an immersed $H^2$-curve $\gamma:[0,1]\to\mathbf{R}^n$ with a point of multiplicity $k$ such that
  \begin{equation}\label{eq:rigidity_open}
    \bar{B}[\gamma]= \varpi^*(k-1)^2.
  \end{equation}
  In addition, equality \eqref{eq:rigidity_open} is attained if and only if $\gamma$ is a $(k-1)$-leafed elastica.
\end{theorem}

\begin{proof}
  The existence of an optimal curve follows since a $\frac{k-1}{2}$-fold figure-eight elastica attains equality.
  In addition, any $(k-1)$-leafed elastica attains \eqref{eq:rigidity_open} by Proposition \ref{prop:regularityenergy_leafedelastica}.

  We now prove rigidity.
  Suppose that $\gamma$ attains \eqref{eq:rigidity_open}.
  Then, as in the proof of Theorem \ref{thm:Li-Yau_open}, the curve $\gamma$ can be divided into $k$ curves $\gamma_1,\dots,\gamma_{k-1}$.
  In addition, equality holds for all the inequalities in the proof of Theorem \ref{thm:Li-Yau_open}.
  In view of the HM-AM inequality \eqref{eq02} we have $L[\gamma_1]=\dots=L[\gamma_{k-1}]$.
  In addition, in view of \eqref{eq01} we also have $L[\gamma_i]B[\gamma_i]=\varpi^*$ for all $i$, and hence by Proposition \ref{prop:minimizer} each curve $\gamma_i$ needs to be a half-fold figure-eight elastica.
  This means that $\gamma$ is a $(k-1)$-leafed elastica, and thus completes the proof.
\end{proof}

Theorem \ref{thm:rigidity_closed} is now a direct consequence of Theorem \ref{thm:rigidity_open} and the contents in Section \ref{subsect:leafedelastica}.

\begin{proof}[Proof of Theorem \ref{thm:rigidity_closed}]
  The existence of a closed curve in $\mathbf{R}^n$ with multiplicity $k$ satisfying \eqref{eq:rigidity_closed} is ensured by Proposition \ref{prop:existence_closedleafedelastica} combined with Proposition \ref{prop:regularityenergy_leafedelastica}.
  Also, we deduce that any of such closed curves must be a $k$-leafed elastica by opening the given closed curve at a multiplicity point and applying Theorem \ref{thm:rigidity_open} to the opened curve.
  Finally, any closed $k$-leafed elastica attains \eqref{eq:rigidity_closed} by Proposition \ref{prop:regularityenergy_leafedelastica}.
\end{proof}

We finally prove Theorem \ref{thm:nonoptimality}.
Proposition \ref{prop:nonexistenceleaf} combined with Theorem \ref{thm:rigidity_open} is already sufficient for asserting the (weaker) statement that there exists no closed $H^2$-curve in $\mathbf{R}^2$ that has a point of multiplicity $k$ and attains \eqref{eq:rigidity_closed}.
However, this does not rule out existence of a minimizing sequence such that $\bar{B}[\gamma_j]\to\varpi^*k^2$.
In order to rule out this phenomenon we ensure general existence of planar optimal curves (which are not necessarily leafed elasticae).

Let $C_k$ denote the class of all immersed closed $H^2$-curves in the plane $\mathbf{R}^2$ with a point of multiplicity $k$.
Let
\begin{equation*}
  \beta_k:=\inf_{\gamma\in C_k}\bar{B}[\gamma].
\end{equation*}
We have $\beta_1=4\pi^2$ obviously.
Theorems \ref{thm:Li-Yau_closed} and \ref{thm:rigidity_closed} imply that $\beta_k=\varpi^*k^2$ for any even $k$.
\if0
For any odd $k\geq3$ we still have
$$\varpi^*k^2\leq\beta_k\leq\varpi^*(k+1)^2.$$
Indeed, the former directly follows by Theorem \ref{thm:Li-Yau_closed} and the last trivial upper bound follows due to the $\frac{k+1}{2}$-fold figure-eight elastica having multiplicity $k$ (in fact $k+1$).
In fact, we have a slightly improved (non-optimal) upper bound for odd $k=2\ell+1$:
$$\beta_k=\beta_{2\ell+1}\leq (2\sqrt{\varpi^*}\ell+2\pi)^2= \big( \sqrt{\varpi^*}(k-1)+2\pi \big)^2.$$
This can be obtained by wrapping a unit-length figure-eight elastica $\ell$ times and attaching a circle of radius $R$ to its center; the resulting curve $\gamma$ is a closed $H^2$-curve with multiplicity $k$, whose energy can be explicitly computed, $\bar{B}[\gamma]=(\ell\cdot1+2\pi R)(\ell\cdot4\varpi^*+2\pi/R)$, so that optimizing in $R$ yields the desired upper bound.
\fi
Although for an odd $k\geq3$ the exact value of $\beta_k$ remains open, here we prove that at least there exists a minimizer attaining the infimum in $\beta_k$, which can be decomposed into $k$ elasticae (this decomposition is however not used later).

\begin{proposition}\label{prop:existence_planarodd}
  For any $k$ (in particular, any odd $k\geq3$) there exists a curve $\bar{\gamma}\in C_k$ such that $\bar{B}[\bar{\gamma}]=\beta_k$.
  In addition, the curve $\bar{\gamma}$ can be divided into $k$ (open) curves $\bar{\gamma}_1,\dots,\bar{\gamma}_k$ at a point of multiplicity $k$, and each curve $\bar{\gamma}_i$ is an elastica satisfying \eqref{eq:elastica} with the same multiplier $\lambda>0$ (not depending on $i=1,\dots,k$).
\end{proposition}

\begin{proof}
  Fix an arbitrary $k$.
  Let $\{\gamma_j\}\subset C_k$ be a minimizing sequence of $\bar{B}$.
  After reparameterization, rescaling, and translation, we may assume that for each $j$ the curve $\gamma_j$ is of unit-speed and of unit-length $L[\gamma_j]=1$, and also there are $0=a_j(1)<a_j(2)<\dots<a_j(k+1)=1$ ($=a_j(1)$) in $\mathbf{T}^1$ such that $\gamma_j(a_j(i))=0$ for all $i$ (thanks to multiplicity $k$).
  Then a standard direct method argument, which is parallel to the proof of Proposition \ref{prop:minimizer}, implies the existence of a unit-speed curve $\bar{\gamma}\in H^2(\mathbf{T}^1;\mathbf{R}^2)$ such that $\beta_k=\bar{B}[\bar{\gamma}]$ and $\gamma_j\to\bar{\gamma}$ in $H^2$-weakly and $C^1$.

  Now we ensure that the limit curve $\bar{\gamma}$ still possesses multiplicity $k$.
  Using the above notation $a_j(i)$, we prove that all adjacent points $a_j(i)$ and $a_j(i+1)$ do not collide as $j\to\infty$.
  Let $L_j(i):=|a_j(i+1)-a_j(i)|$.
  By the unit-speed parameterization, $L_j(i)$ is nothing but the length of the curve $\gamma_j|_i:=\gamma_j|_{[a_j(i),a_j(i+1)]}$, and hence by the Cauchy-Schwarz inequality $\bar{B}=LB\geq TC^2$ involving the total curvature $TC[\gamma]=\int_\gamma|\kappa|ds$ we have
  \begin{equation*}
    B[\gamma_j] \geq B[\gamma_j|_i] \geq \frac{1}{L_j(i)}TC[\gamma_j|_i]^2 \geq \frac{1}{L_j(i)}\pi^2,
  \end{equation*}
  where the last estimate $TC[\gamma_j|_i]\geq\pi$ follows by a generalization of Fenchel's theorem, namely by Lemma \ref{lem:piecewise} with $N=1$.
  By the boundedness of $B[\gamma_j]$ we deduce that there is $\delta>0$ such that $L_j(i)\geq\delta$ holds for all $i$ and $j$.
  This means that no adjacent points collide, and hence up to a subsequence (without relabeling), all $a_j(1),\dots,a_j(k)$ converge to distinct $k$ points $a(1),\dots,a(k)$ in $\mathbf{T}^1$ as $j\to\infty$.
  By $C^1$-convergence we have $\bar{\gamma}(a(i))=0$ for all $i=1,\dots,k$.
  Therefore, $\bar{\gamma}$ still has multiplicity $k$ so that $\bar{\gamma}\in C_k$.
  This ensures the existence of a minimizer.

  Finally, we prove that any minimizer $\bar{\gamma}\in C_k$ can be divided into $k$ elasticae with a same positive multiplier $\lambda>0$.
  In fact, if $\bar{B}[\bar{\gamma}]=\beta_k$, then after rescaling by $\Lambda>0$ so that $L[\Lambda\bar{\gamma}]=B[\Lambda\bar{\gamma}]=\sqrt{\beta_k}$, the curve $\Lambda\bar{\gamma}$ is a minimizer of the functional $E:=E_1=B+L$ since $E\geq 2\bar{B}^\frac{1}{2}$ and equality is attained for $\Lambda\bar{\gamma}$.
  In particular, if we let $p\in\mathbf{R}^2$ be a point of multiplicity $k$ and choose (ordered) $k$ distinct points $a_1,\dots,a_k\in\bar{\gamma}^{-1}(p)$, and if we divide $\bar{\gamma}$ into $k$ (open) curves $\bar\gamma_1,\dots,\bar\gamma_k$ by cutting at $a_1,\dots,a_k\in\mathbf{T}^1$, then the minimality of $\Lambda\bar{\gamma}$ for $E$ implies that each $\Lambda\bar{\gamma}_i$ satisfies \eqref{eq:elastica} with $\lambda=1$.
  Hence, going back to the original scale, we find that $\bar{\gamma}_i$ satisfies \eqref{eq:elastica} with $\lambda=\Lambda^2$.
  This ensures the desired decomposition of $\bar{\gamma}$ into $k$ elasticae.
\end{proof}

We are now in a position to complete the proof of Theorem \ref{thm:nonoptimality}.

\begin{proof}[Proof of Theorem \ref{thm:nonoptimality}]
  We argue by contradiction.
  Suppose that for an odd $k\geq3$ the assertion does not hold.
  Then there exists a sequence $\{\gamma_j\}\subset C_k$ such that $\bar{B}[\gamma_j]\to\varpi^*k^2$; this means that $\beta_k=\varpi^*k^2$ since $\varpi^*k^2$ is a lower bound, cf.\ Theorem \ref{thm:Li-Yau_closed}.
  By Proposition \ref{prop:existence_planarodd} there exists a closed curve $\bar{\gamma}\in C_k$ attaining equality \eqref{eq:rigidity_closed}.
  Then by applying Theorem \ref{thm:rigidity_open} (as in the proof of Theorem \ref{thm:rigidity_closed}) we conclude that $\bar{\gamma}$ must be a closed $(2,k)$-leafed elastica.
  However, this contradicts the nonexistence result in Proposition \ref{prop:nonexistenceleaf}.
\end{proof}

We close this section by mentioning miscellaneous remarks.
%In what follows we mention more on rigidity of closed curves attaining equality in \eqref{eq:Li-Yau_closed}, or equivalently of closed leafed elasticae, under additional assumptions.

\begin{remark}[Non-uniqueness of closed leafed elasticae]\label{rem:nonuniqueness}
  In contrast to the fact that generic uniqueness of closed $k$-leafed elasticae holds for $k\leq3$, cf.\ Propositions \ref{prop:uniqunessleaf12} and \ref{prop:uniquenessleaf3}, this is not the case for $k\geq4$.
  Indeed, for any $n\geq2$ and $m\geq1$, if an $(n,2m)$-leafed elasticae consists of $m$ figure-eight elasticae, then there is an $\mathbf{S}^{n-2}$-freedom around the joint-axis; if $n\geq3$, then such rotation-type non-uniqueness phenomena occur for any number $k\geq4$ of leaves, cf.\ Example \ref{ex:propellereight}.
  In addition, there is another mechanism of non-uniqueness; for any $k\geq4$ there remains a certain freedom for gluing $k$ papers of congruent obtuse isosceles triangles (with angle $2\phi^*>\pi/2$, cf.\ Lemma \ref{lem:angle}) along their equal-length sides even in the three-dimensional space (in contrast to $k=3$).
  See Figure \ref{fig:freedom} for examples in the case of $k=4$.
\end{remark}

\begin{figure}[htbp]
  \includegraphics[width=100mm]{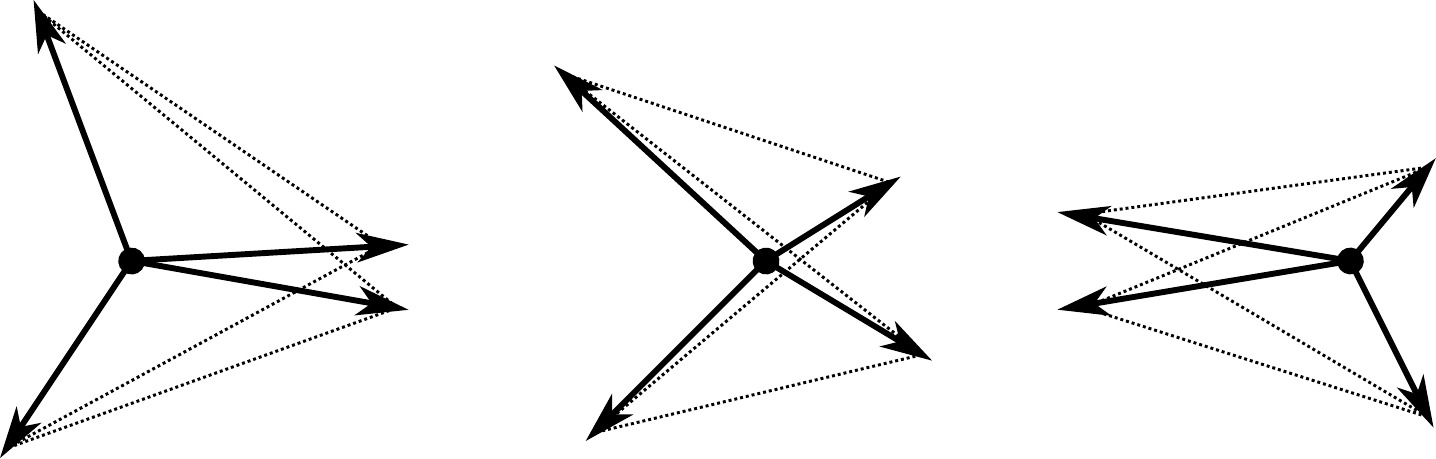}
  \caption{Examples of the elements in $\Omega^*(3,4)$ in Lemma \ref{lem:closedleafedcharacterization}.}
  \label{fig:freedom}
\end{figure}

\begin{remark}[Uniqueness of leafed elasticae under $C^3$-regularity]
  By representation \eqref{eq:figureeightcurvature} and derivative formula \eqref{eq:derivativecn} we easily deduce that if a $k$-leafed elastica $\gamma$ is of class $C^3([0,1];\mathbf{R}^n)$, then the curve $\gamma$ is a unique planar curve given by a $\frac{k}{2}$-fold figure-eight elastica.
  As a consequence, if $k\geq2$ is even, and if $\gamma$ is a closed $(n,k)$-leafed elastica of class $C^3(\mathbf{T}^1;\mathbf{R}^n)$, then $\gamma$ is a unique planar curve (up to similarity and reparameterization) given by a $\frac{k}{2}$-fold closed figure-eight elastica; also, if $k\geq3$ is odd, then there exists no closed $(n,k)$-leafed elastica of class $C^3(\mathbf{T}^1;\mathbf{R}^n)$.
\end{remark}

%Let $k\geq4$. Then there exists no $k$-fold rotationally symmetric closed $(3,k)$-leafed elastica.

\begin{remark}[Total curvature]
  As is indicated in \cite{Mueller2021}, the total (absolute) curvature $TC[\gamma]=\int_\gamma|\kappa|ds$ is not an effective embeddedness criterion for closed curves since the total curvatures of both an embedded thin convex curve and a thin figure-eights (closed to a segment) are nearly $2\pi$.
  However, we still have the optimal lower bound $TC[\gamma]>(k-1)\pi$ for open curves with multiplicity $k$ via Fenchel's theorem (as in the proof of Proposition \ref{prop:existence_planarodd}) and hence $TC[\gamma]>k\pi$ for closed curves with multiplicity $k$.
\end{remark}

\section{Elastic flows}\label{sect:elasticflow}

In this last section we discuss applications to elastic flows.
We call the $L^2$-gradient flow of the energy $E_\lambda:=B+\lambda L$ (as in \eqref{eq:energy_E_lambda}) for a given $\lambda>0$ {\em elastic flow}, and that of the bending energy $B$ under the fixed-length constraint $L[\gamma]=L_0>0$ {\em fixed-length elastic flow}.
Such flows are given by one-parameter families of curves $\gamma:\mathbf{T}^1\times[0,\infty)\to\mathbf{R}^n$ solving the fourth order PDE in the form of
\begin{equation}\label{eq:elasticflow}
  \partial_t\gamma = -2\nabla_s^2\kappa-|\kappa|^2\kappa+\lambda\kappa,
\end{equation}
where in the former case $\lambda>0$ is a fixed number given in $E_\lambda$, while in the latter case it depends on the solution and is given in the form of
\begin{equation}
  \lambda(t)=\frac{\int_{\gamma(t)}\langle 2\nabla_s^2\kappa+|\kappa|^2\kappa,\kappa \rangle ds}{\int_{\gamma(t)}|\kappa|^2ds}.
\end{equation}
At least since Wen's 1995 paper \cite{Wen1995} and Polden's 1996 thesis \cite{Polden1996}, elastic flows have been studied by many authors, see e.g.\ the recent nice survey \cite{Mantegazza2021} and references therein.
Concerning these flows, long-time existence and smooth convergence to an elastica are valid in general at least from smooth closed initial curves.
Those results follow by combining the fundamental result by Dziuk--Kuwert--Sch\"{a}tzle \cite{Dziuk2002} with recent developments on the {\L}ojasiewicz-Simon inequality as is demonstrated in \cite{Mueller2021} (see also \cite{Mantegazza2021}); we note that the argument in \cite{Mueller2021} directly works for higher codimensions as so do the key ingredients \cite{Dziuk2002,Rupp2020,Mantegazza2021b,Rupp2020a}.

Our Li--Yau type inequality can be used for ensuring %all-time embeddedness of solutions below certain energy thresholds.
embeddedness of solutions for all $t\geq0$ below certain energy thresholds.
Note that in second-order flows such a property holds generically (without smallness) by the maximum principle (see e.g.\ \cite{Huisken1998,Brendle2014}), but this is not the case for higher-order flows including elastic flows \cite{Blatt2010}.

\begin{theorem}\label{thm:elasticflow}
  Let $\lambda>0$ and $\gamma_0:\mathbf{T}^1\to\mathbf{R}^n$ be a closed curve such that $\frac{1}{4\lambda}E_\lambda[\gamma_0]^2<4\varpi^*$.
  Then the elastic flow starting from $\gamma_0$ is embedded for all $t\geq0$.
  In addition, it smoothly converges as $t\to\infty$ to a one-fold round circle of radius $\frac{1}{\sqrt{2\lambda}}$ up to reparameterization.
\end{theorem}

\begin{theorem}\label{thm:fixedlengthelasticflow}
  Let $\gamma_0:\mathbf{T}^1\to\mathbf{R}^n$ be a closed curve such that $\bar{B}[\gamma_0]<4\varpi^*$.
  Then the fixed-length elastic flow starting from $\gamma_0$ is embedded for all $t\geq0$.
  In addition, it smoothly converges as $t\to\infty$ to a one-fold round circle of radius $\frac{L[\gamma_0]}{2\pi}$ up to reparameterization.
\end{theorem}

These results extend M\"{u}ller--Rupp's corresponding results in \cite{Mueller2021} from $n=2$ to $n\geq2$.
%All-time embeddedness and convergence to a circle are contributions of the smallness assumptions.
The threshold $4\varpi^*$ is optimal simultaneously for all-time embeddedness and for convergence to a circle; indeed, to each flow, a figure-eight elastica of suitable size is a non-embedded stationary solution and attains the threshold $4\varpi^*$.
We remark that if we include embeddedness of an initial curve in the assumption, then the threshold value is significantly improved when $n=2$, see our subsequent work \cite{MiuraMuellerRupp}.

We may safely omit the proof of the above theorems since M\"{u}ller-Rupp \cite{Mueller2021} already provide detailed proofs that completely work for $n\geq3$ {\em once} our result (Theorem \ref{thm:Li-Yau_closed}) is established.
A key point is that $\bar{B}\leq\frac{1}{4\lambda}E_\lambda^2$ holds since $4\lambda ab\leq(a+\lambda b)^2$, and hence $\bar{B}<4\varpi^*$ holds for all time by the gradient-flow structure.
Uniqueness of the limit profile follows since the circle is the only closed elastica such that $\bar{B}<4\varpi^*$; this can be verified by energy quantization of closed elasticae.
(Similar arguments have been previously used for the Willmore flow \cite{Kuwert2004} through the Li--Yau inequality \cite{Li1982} and Bryant's classification of Willmore spheres \cite{Bryant1984}.)

We argue more on the energy quantization of closed elasticae.
Thanks to the classification of closed elasticae, cf.\ \cite{Langer1984}, any planar closed elastica is an $N$-fold circle (with energy $\bar{B}=4\pi^2N^2$) or an $N$-fold figure-eight elastica (with $\bar{B}=4\varpi^*N^2$, cf.\ Lemma \ref{lem:figureeight} (iii)), while any non-planar closed elastica is an embedded nontrivial torus knot or its multiple covering, all of which satisfy $\bar{B}[\gamma]>16\pi^2$ by the classical F\'{a}ry-Milnor theorem $TC[\gamma]>4\pi$, cf.\ \cite{Fary1949,Milnor1950}.
Therefore, in order to show that
\begin{center}
  $\gamma$ is a closed elastica such that $\bar{B}[\gamma]<4\varpi^*$ $\Longrightarrow$ $\gamma$ is a circle,
\end{center}
it is sufficient to check that
$4\varpi^*<16\pi^2.$
This is already proved in \cite{Sachkov2012,Mueller2021} by estimating elliptic integrals directly.
This also follows variationally, since our key %Lemma \ref{lem:closedleafedcharacterization}
Proposition \ref{prop:minimizer} combined with the fact that a circle belongs to $X_0$ and has energy $\bar{B}=4\pi^2$ implies that $\varpi^*<4\pi^2$.
Here we give an alternative variational proof which only relies on the classification of planar closed elasticae:

\begin{remark}[Variational argument for $4\varpi^*<16\pi^2$]
  Let $Z_0$ be the class of zero-rotation-number planar closed $H^2$-curves of unit-length.
  Since $Z_0$ is closed in $H^2$-weak (or $C^1$) and open in $H^2$, by a direct method and bootstrap argument there is a smooth minimizer, which is by classification a figure-eight elastica, and hence $\min_{Z_0}B=4\varpi^*$.
  On the other hand, another ``figure-eight'' curve $\tilde{\gamma}\in Z_0$ made by osculating two circles has energy $B[\tilde{\gamma}]=16\pi^2>\min_{Z_0}B=4\varpi^*$.
\end{remark}

The energy quantization observed here is summarized as follows:
\begin{proposition}
  Let $C$ be the set of all closed elasticae in $\mathbf{R}^n$, $n\geq3$.
  Then there is a strictly increasing sequence $\{b_k\}_{k=1}^{\infty}\subset[4\pi^2,\infty)$ such that $\bar{B}(C)=\{b_1,b_2,\dots\}$ and that $b_1=4\pi^2$, $b_2=4\varpi^*$, and $b_3=16\pi^2$.
  In addition, the preimage $\bar{B}^{-1}(b_1)$ consists of circles, $\bar{B}^{-1}(b_2)$ figure-eight elasticae, and $\bar{B}^{-1}(b_3)$ two-fold circles.
\end{proposition}

We finally discuss elastic flows of open curves (cf.\ references in \cite{Mantegazza2021}), where we may also obtain similar thresholds as in the above theorems (with $4\varpi^*$ replaced by $\varpi^*$).
Such thresholds are particularly effective for the zero Navier boundary condition, i.e., 
$\kappa(0)=\kappa(1)=0$, as there always exist admissible initial curves below $\varpi^*$.
However, for the clamped boundary condition, i.e., $\gamma(0)=P_0$, $\gamma(1)=P_1$, $\gamma_s(0)=V_0$, $\gamma_s(1)=V_1$ for $P_0,P_1,V_0,V_1\in\mathbf{R}^n$ (with $|V_0|=|V_1|=1$), it may happen that no admissible curve exists below the threshold.
It is also difficult to detect convergent limits; indeed, even global minimizers may not be unique, and it is a quite delicate issue to seek an effective range in which uniqueness holds; see \cite{Miura2020} for details.
The same discussion is also valid for the fixed-length case.

\section{Elastic networks}\label{sect:elasticnetwork}

Recently many studies have been devoted to understanding the geometric nature of networks.
This is also the case for elastic curves, see e.g.\ \cite{Barrett2012,DallAcqua2017,DallAcqua2019,Garcke2019,DallAcqua2020,Garcke2020,Novaga2020}.

In this section, by applying our key inequality, Proposition \ref{prop:minimizer}, we prove that among so-called $\Theta$-networks in $\mathbf{R}^n$ there exists a minimizer of the energy $E:=E_1=B+L$, thus extending \cite[Theorem 4.10]{DallAcqua2020} to a general codimension.
Note that our problem is equivalent to minimizing $E_\lambda=B+\lambda L$ or $\bar{B}=LB$ up to rescaling; we choose $E$ just for the sake of compatibility with \cite{DallAcqua2020}.

A triplet of immersed $H^2$-curves $\Gamma=(\gamma_1,\gamma_2,\gamma_3)\in(H^2(I;\mathbf{R}^n))^3$ is called {\em ($n$-dimensional) $\Theta$-network} if the endpoints of the curves meet at triple junctions, i.e., $\gamma_1(0)=\gamma_2(0)=\gamma_3(0)$ and $\gamma_1(1)=\gamma_2(1)=\gamma_3(1)$, and in addition if the curves meet at the triple junctions with equal angles of $\frac{2\pi}{3}$.
%(and hence the junctions are asymptotically planar).
Let $\Theta(\mathbf{R}^n)$ denote the class of all $n$-dimensional $\Theta$-networks.
For any $\Gamma=(\gamma_1,\gamma_2,\gamma_3)\in\Theta(\mathbf{R}^n)$ we define
$$E[\Gamma] := \sum_{i=1}^3E[\gamma_i] = \sum_{i=1}^3\int_{\gamma_i}(|\kappa|^2+1)ds.$$

The main result in this section is the following

\begin{theorem}[Existence of minimal elastic $\Theta$-networks]\label{thm:existence_network}
  Let $n\geq2$.
  Then there exists an $n$-dimensional $\Theta$-network $\bar{\Gamma}\in\Theta(\mathbf{R}^n)$ such that
  \begin{equation*}
    E[\bar{\Gamma}]=\inf_{\Gamma\in\Theta(\mathbf{R}^n)}E[\Gamma].
  \end{equation*}
\end{theorem}

This result is not a straightforward consequence of the direct method since the set $\Theta(\mathbf{R}^n)$ has no appropriate compactness in general, in the sense that a component-curve may degenerate into a point under the boundedness of the energy.
Such a phenomenon can be however ruled out for a minimizing sequence by showing that the minimal energy among ``degenerate'' networks is greater than the energy of a certain (nondegenerate) $\Theta$-network.
In fact, Dall'Acqua--Novaga--Pluda \cite{DallAcqua2020} demonstrate that this strategy successfully works at least in the planar case $n=2$ (see also \cite{DallAcqua2021} for a more detailed proof), using a computer-assisted argument in the middle.
Our argument here extends their result to $n\geq2$, and also provides a non-computer-assisted analytic proof, cf.\ Lemma \ref{lem:degenerate_network} below.

Now we enter the proof of Theorem \ref{thm:existence_network}.

We first indicate that Fenchel's theorem can be extended to piecewise smooth closed curves.
The proof is given in Appendix \ref{sec:app:Fenchel}.

\begin{lemma}\label{lem:piecewise}
  Let $\gamma_1,\dots,\gamma_N\in W^{2,1}(I;\mathbf{R}^n)\subset C^1(\bar{I};\mathbf{R}^n)$ be immersed curves such that $\gamma_{j}(1)=\gamma_{j+1}(0)=p_j$, where we interpret $\gamma_{N+1}:=\gamma_1$.
  For all $j=1,\dots,N$ let $\theta_j\in[0,\pi]$ denote the external angle at the vertex $p_j$, i.e.,
  $$\cos\theta_j=\left\langle \frac{\gamma_{j}'(1)}{|\gamma_{j}'(1)|},\frac{\gamma_{j+1}'(0)}{|\gamma_{j+1}'(0)|} \right\rangle.$$
  Then
  $$\sum_{j=1}^NTC[\gamma_j]=\sum_{j=1}^N\int_{\gamma_j}|\kappa|ds \geq 2\pi -\sum_{i=1}^N\theta_i.$$
\end{lemma}

Using this lemma, we prove the following key dichotomy result as in \cite{DallAcqua2020}:

\begin{lemma}\label{lem:dichotomy}
  Let $\{\Gamma_j\}_j=\{(\gamma_{1,j},\gamma_{2,j},\gamma_{3,j})\}_j\subset\Theta(\mathbf{R}^n)$ be a sequence such that $\sup_{j}E[\Gamma_j]<\infty$.
  Then, up to reparameterization, translation, and taking a subsequence (all without relabeling), one of the following mutually exclusive assertions holds:
  \begin{enumerate}
    \item There exists a $\Theta$-network $\Gamma=(\gamma_1,\gamma_2,\gamma_3)\in\Theta(\mathbf{R}^n)$ such that for each $i=1,2,3$, the sequence $\{\gamma_{i,j}\}_j$ converges to $\gamma_i$ as $j\to\infty$ in the $H^2$-weak and $C^1$ topology.
    In particular, $\liminf_{j\to\infty}E[\Gamma_j] \geq E[\Gamma]$.
    \item Up to permutations of the index $i$, we have $L[\gamma_{1,j}]\to0$ as $j\to\infty$.
    In addition, there are two immersed $H^2$-curves $\gamma_2,\gamma_3:I\to\mathbf{R}^n$ such that $\gamma_2(0)=\gamma_2(1)=\gamma_3(0)=\gamma_3(1)$ and such that the sequence $\{\gamma_{2,j}\}_j$ (resp.\ $\{\gamma_{3,j}\}_j$) converges to $\gamma_2$ (resp.\ $\gamma_3$) as $j\to\infty$ in the both $H^2$-weak and $C^1$ topology.
    In particular, $\liminf_{j\to\infty}E[\Gamma_j] \geq E[\gamma_2]+E[\gamma_3]$.
  \end{enumerate}
\end{lemma}

\begin{remark}
  The latter case corresponds to ``degenerate'' networks.
  In fact we have an additional constraint on the angles of $\gamma_2$ and $\gamma_3$ at their endpoints due to the original angle condition for $\Theta$-networks, cf.\ \cite[Definition 3.1]{DallAcqua2020}, but here (and there) such a constraint is not used.
\end{remark}

\begin{proof}[Proof of Lemma \ref{lem:dichotomy}]
  Throughout the proof we may suppose that after reparameterization, each curve $\gamma_{i,j}$ has domain $\bar{I}=[0,1]$ and is of constant speed, and also after translation, $\gamma_{i,j}(0)=0$.
  The proof proceeds similarly to \cite{DallAcqua2020}.

  We first consider the case that $\inf_j\min_{i=1,2,3}L[\gamma_{i,j}]>0$.
  Then the assumptions of constant-speed and energy-boundedness imply that
  \begin{equation*}
     \sum_{i=1,2,3}\frac{1}{L[\gamma_{i,j}]^3}\|\gamma_{i,j}''\|_{L^2(I;\mathbf{R}^n)}^2 = \sum_{i=1,2,3}B[\gamma_{i,j}] \leq  \sup_{j}E[\Gamma_j]<\infty.
  \end{equation*}
  By the fact that $\sup_{j}\max_{i=1,2,3}L[\gamma_{i,j}]\leq\sup_{j}E[\Gamma_j]<\infty$ we deduce the $L^2$-boundedness of $\{\gamma''_{i,j}\}_j$ for each $i=1,2,3$.
  Then, noting the nondegeneracy assumption on length, we deduce from a parallel argument to the proof of Proposition \ref{prop:existence_planarodd} that up to a subsequence each sequence $\{\gamma_{i,j}\}_j$ converges in the desired sense, and in particular $C^1$-convergence ensures that the limit curves again form a $\Theta$-network, so that assertion (i) holds.

  Next, we consider the case that, after permutations, $\inf_jL[\gamma_{1,j}]=0$ holds but we still have $\inf_j\min_{i=2,3}L[\gamma_{i,j}]>0$.
  In this case, the same argument as above ensures the desired convergence of $\gamma_{2,j}$ and $\gamma_{3,j}$ so that assertion (ii) holds.

  We finally prove that only the above cases are possible to occur, i.e., no two (or three) components degenerate.
  For each pair of $i,i'\in\{1,2,3\}$ with $i\neq i'$, the $\frac{2\pi}{3}$-angle condition on $\Theta$-networks implies that the curves $\gamma_{i,j}$ and $\gamma_{i',j}$ form a piecewise closed curve with exactly two jumps of angle $\theta_1=\theta_2=\pi/3$, and hence by Lemma \ref{lem:piecewise} we have $TC[\gamma_{i,j}]+TC[\gamma_{i',j}]\geq 4\pi/3$.
  By the Cauchy--Schwarz inequality that $L[\gamma]B[\gamma]\geq TC[\gamma]^2$,
  $$\sum_{k=i,i'}B[\gamma_{k,j}]^\frac{1}{2} \geq \sum_{k=i,i'}\frac{1}{L[\gamma_{k,j}]^\frac{1}{2}}TC[\gamma_{k,j}] \geq \min_{k=i,i'}\frac{1}{L[\gamma_{k,j}]^\frac{1}{2}} \frac{4\pi}{3}.$$
  Then energy-boundedness implies that $\inf_{j}\max_{k=i,i'}L[\gamma_{k,j}]>0$.
  By the arbitrariness of the choice of $i$ and $i'$, up to taking a subsequence, there are at least two indices $i\in\{1,2,3\}$ such that $\inf_{j}L[\gamma_{i,j}]>0$.
\end{proof}

In view of the above lemma we are naturally led to study the energy of ``drops'' appearing in the degenerate case.
The next statement is about energy optimal drops; it is a key ingredient that this estimate holds in any codimension.

\begin{proposition}\label{prop:minimizer2}
  Let $\gamma:[a,b]\to\mathbf{R}^n$ be an immersed $H^2$-curve such that $\gamma(a)=\gamma(b)$.
  Then
  \begin{equation}\label{eq:minimizer2}
    E[\gamma]\geq 2\sqrt{\varpi^*},
  \end{equation}
  where equality is attained if and only if $\gamma$ is a half-fold figure-eight elastica of length $\sqrt{\varpi^*}$.
\end{proposition}

\begin{proof}
  Since $E[\gamma]= B[\gamma]+L[\gamma]\geq 2\sqrt{L[\gamma]B[\gamma]}$, Proposition \ref{prop:minimizer} implies that $E[\gamma]\geq 2\sqrt{\varpi^*}$.
  In addition, equality holds if and only if $B[\gamma]=L[\gamma]$ and $L[\gamma]B[\gamma]=\varpi^*$; in particular, $L[\gamma]=\sqrt{\varpi^*}$.
  Rigidity in Proposition \ref{prop:minimizer} implies that $\gamma$ is a half-fold figure-eight elastica, completing the proof.
\end{proof}

Now the main issue is reduced to showing that there is a (nondegenerate) $\Theta$-network of less energy than the minimal energy of degenerate networks consisting of two drops.
More precisely, we prove

\begin{lemma}\label{lem:degenerate_network}
  There exists a planar $\Theta$-network $\Gamma\in\Theta(\mathbf{R}^2)$ such that $E[\Gamma] < 4\sqrt{\varpi^*}$.
\end{lemma}

In \cite{DallAcqua2020} the test network $\Gamma$ is taken to be a standard double-bubble, but the desired estimate is shown only through computer-assisted numerical computations, cf.\ \cite[Lemma 4.9]{DallAcqua2020} borrowed from \cite[Proposition 6.4]{DallAcqua2017}.
Here we provide a completely analytic proof by a totally different approach.
In fact, we construct a competitor $\Gamma$ by using a piece of a wavelike elastica, the energy estimate of which is based on monotonicity along deformations of wavelike elasticae with respect to parameter.

\begin{figure}[htbp]
  \includegraphics[width=30mm]{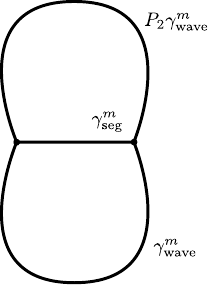}
  \caption{A sketch of $\Gamma_\mathrm{wave}^m=(\gamma_\mathrm{wave}^m,P_2\gamma_\mathrm{wave}^m,\gamma_\mathrm{seg}^m)$.}
  \label{fig:Gamma_wave}
\end{figure}

\begin{proof}[Proof of Lemma \ref{lem:degenerate_network}]
  Let $C_\mathrm{conv}$ be the class of all locally convex planar curves $\zeta=(\zeta^1,\zeta^2)^\top:[-a,a]\to\mathbf{R}^2$ such that $\zeta^1(x)=-\zeta^1(-x)$ and $\zeta^2(x)=\zeta^2(-x)\leq0$ hold for any $x\in[-a,a]$, and such that $\pm\zeta^1(\pm a)>0$ and $\zeta^2(\pm a)=0$.
  Below we construct a planar competitor $\Gamma=(\gamma_1,\gamma_2,\gamma_3)\in\Theta(\mathbf{R}^2)$ with the property that $\gamma_1\in C_\mathrm{conv}$, $\gamma_2=P_2\gamma_1$, where $P_2$ denotes the refection with respect to the $e_1$-axis, and $\gamma_3$ is a segment along the $e_1$-axis.

  Recall (from Section \ref{subsect:elastica} or \cite{Mueller2021}) that a half-period of a wavelike elastica in the plane $\mathbf{R}^2$ is given by $\gamma_\mathrm{wave}^m(s)$ in \eqref{eq:wavelike} with $s\in[-K(m),K(m)]$.
  Then the same computation in the proof of Lemma \ref{lem:figureeight} (iii) shows that
  \begin{equation}\label{eq08}
    L[\gamma_\mathrm{wave}^m]=2K(m), \quad B[\gamma_\mathrm{wave}^m]=8\big(E(m)-(1-m)K(m)\big).
  \end{equation}
  In addition, since the value of the first component of $\gamma_\mathrm{wave}^m$ at $s=\pm K(m)$ is given by $(\gamma_\mathrm{wave}^m)^1(\pm K(m))=\pm (2E(m)-K(m))$, and since $2E(m)-K(m)>0$ holds for any $m<m^*$ (by definition of $m^*$ and monotonicity of $2E-K$), we have $\gamma_\mathrm{wave}^m\in C_\mathrm{conv}$ for $m<m^*$.

  We then define (up to reparameterization) a network $\Gamma_\mathrm{wave}^m$ for $m<m_*$ by
  $$\Gamma_\mathrm{wave}^m:=(\gamma_\mathrm{wave}^m,P_2\gamma_\mathrm{wave}^m,\gamma_\mathrm{seg}^m),$$
  where $\gamma_\mathrm{seg}^m(x):=\big( (2E(m)-K(m))x,0 \big)^\top$ for $x\in(-1,1)$, cf.\ Figure \ref{fig:Gamma_wave}.
  (Note that $\Gamma_\mathrm{wave}^m$ is not necessarily a $\Theta$-network since the $\frac{2\pi}{3}$-angle condition at the junctions may not be satisfied.)
  By \eqref{eq08}, and by $L[\gamma_\mathrm{seg}^m]=2(2E(m)-K(m))$ and $B[\gamma_\mathrm{seg}^m]=0$, we have
  \begin{align*}
    L[\Gamma_\mathrm{wave}^m] &= 2L[\gamma_\mathrm{wave}^m] + L[\gamma_\mathrm{seg}^m] = 2\big( 2E(m)+K(m) \big),\\
    B[\Gamma_\mathrm{wave}^m] &= 2B[\gamma_\mathrm{wave}^m] + B[\gamma_\mathrm{seg}^m] = 16\big(E(m)-(1-m)K(m)\big).
  \end{align*}
  Then, after rescaling so that $E=2\bar{B}^\frac{1}{2}$, namely taking
  $$\hat\Gamma_\mathrm{wave}^m:=\Lambda_m\Gamma_\mathrm{wave}^m$$
  with $\Lambda_m:=\sqrt{B[\Gamma_\mathrm{wave}^m]/L[\Gamma_\mathrm{wave}^m]}$, we have
  \begin{equation}\label{eq09}
    E[\hat{\Gamma}_\mathrm{wave}^m] = 2\bar{B}[\Gamma_\mathrm{wave}^m]^\frac{1}{2} = 2\sqrt{32\big( 2E(m)+K(m) \big)\big(E(m)-(1-m)K(m)\big)}.
  \end{equation}
  In particular, by this representation and by definition of $m^*$ and $\varpi^*$,
  \begin{equation}\label{eq10}
    \lim_{m\to m^*}E[\hat{\Gamma}_\mathrm{wave}^m] = 4\sqrt{\varpi^*}.
  \end{equation}

  We now prove that the above-defined network $\hat{\Gamma}_\mathrm{wave}^m$ with parameter $m=\frac{3}{4}<m^*$, cf.\ Lemma \ref{lem:parameter}, gives the desired $\Theta$-network.
  More precisely, we prove:
  \begin{enumerate}
    \item $\hat{\Gamma}_\mathrm{wave}^{3/4}\in\Theta(\mathbf{R}^2)$, i.e., satisfies the $\frac{2\pi}{3}$-angle condition at the junctions.
    \item $E[\hat{\Gamma}_\mathrm{wave}^{3/4}]<4\sqrt{\varpi^*}$.
  \end{enumerate}

  We first prove property (i).
  It suffices to consider $\Gamma_\mathrm{wave}^{3/4}$ (before rescaling).
  In addition, by symmetry we only need to prove that if we let $\phi_m\in(0,\pi)$ denote the angle made by the tangent vector of $\gamma_\mathrm{wave}^{3/4}$ at the endpoint $s=-K(m)$ and the vector $-e_1=(-1,0)^\top$, then $\phi_{3/4}=\pi/3$.
  Arguing similarly to the proof of Lemma \ref{lem:figureeight} (iv), we obtain
  \begin{equation}\label{eq12}
    \cos\phi_m = 2m-1.
  \end{equation}
  In particular, $\cos\phi_{3/4}=1/2$ and hence $\phi_{3/4}=\pi/3$.

  We finally prove property (ii).
  To this end, by \eqref{eq09} and \eqref{eq10}, and by the fact that $\frac{3}{4}<m^*$, it is sufficient to prove that the energy in \eqref{eq09} is strictly increasing in $m\in(0,1)$.
  It is thus sufficient to show monotonicity of the functions $f$ and $g$ defined through $h(m):=E(m)-(1-m)K(m)$ by
  $$f(m):=E(m)h(m), \quad g(m):=K(m)h(m).$$
  By the well-known derivative formulae (cf.\ \cite[p.521]{Whittaker1962}) that
  $$\tfrac{d}{dm}E(m)=\tfrac{1}{2m}(E(m)-K(m)), \quad \tfrac{d}{dm}K(m)=\tfrac{1}{2m(1-m)}(E(m)-(1-m)K(m)),$$
  we find that $E'=\tfrac{1}{2m}h-\tfrac{1}{2}K$ and $K'=\tfrac{1}{2m(1-m)}h$, and also that $h'=\frac{1}{2}K$.
  Hence,
  %$$h'(m)=\tfrac{1}{2m}(E-K)+K-\tfrac{1}{2m}(E-(1-m)K) = \tfrac{1}{2}K,$$
  \begin{align*}
    f' = E'h+Eh' &= (\tfrac{1}{2m}h-\tfrac{1}{2}K)h + (h+(1-m)K)(\tfrac{1}{2}K)\\
    &= \tfrac{1}{2m}h^2 + \tfrac{1-m}{2}K^2 >0.\\
    g' = K'h+Kh' &= \tfrac{1}{2m(1-m)}h^2+\tfrac{1}{2}K^2>0.
  \end{align*}
  This implies the desired monotonicity of the energy in \eqref{eq09}.
\end{proof}

We are now ready to complete the proof of Theorem \ref{thm:existence_network}.

\begin{proof}[Proof of Theorem \ref{thm:existence_network}]
  Let $\{\Gamma_j\}_j\subset\Theta(\mathbf{R}^n)$ be a minimizing sequence such that $E[\Gamma_j]\to\inf_{\Theta(\mathbf{R}^n)}E$.
  Then, possibly after reparameterization, translation, and taking a subsequence, either assertion (i) or (ii) in Lemma \ref{lem:dichotomy} holds.
  However, assertion (ii) does not occur in view of the last lower semicontinuity condition that $\liminf_{j\to\infty}E[\Gamma_j]\geq E[\gamma_2]+E[\gamma_3]$; indeed, since each of the curves $\gamma_2$ and $\gamma_3$ in (ii) is a competitor in Proposition \ref{prop:minimizer2}, we have $E[\gamma_2]+E[\gamma_3]\geq4\sqrt{\varpi^*}$; on the other hand, by Lemma \ref{lem:degenerate_network},
  $$\lim_{j\to\infty}E[\Gamma_j]=\inf_{\Theta(\mathbf{R}^n)}E\leq \inf_{\Theta(\mathbf{R}^2)}E<4\sqrt{\varpi^*}.$$
  Therefore we have assertion (i), which implies that there exists a minimizing $\Theta$-network $\Gamma\in\Theta(\mathbf{R}^n)$ since $E[\Gamma]\leq\lim_{j\to\infty}E[\Gamma_j]=\inf_{\Theta(\mathbf{R}^n)}E \leq E[\Gamma]$.
\end{proof}

Concerning minimal elastic $\Theta$-networks, there are many further problems remained open; for example, as is posed in \cite{DallAcqua2017,DallAcqua2020}, symmetry and global injectivity are still open even in the planar case.
Component-wise injectivity as in \cite[Proposition 4.11]{DallAcqua2020} can be shown via computer-assisted estimates also for any $n\geq2$.

\begin{remark}[Generalized $\Theta$-networks]
  In fact, in the proof in Lemma \ref{lem:degenerate_network}, not only the energy $E[\hat{\Gamma}_\mathrm{wave}^m]$ but also the angle $\phi_m$ in \eqref{eq12} is monotone in $m$, and hence the same assertion holds for a more general class of networks; namely, we may replace the symmetric angle condition $(\frac{2\pi}{3},\frac{2\pi}{3},\frac{2\pi}{3})$ by the partially asymmetric condition $(\alpha,\alpha,2\pi-2\alpha)$ with $\alpha\in(0,\pi-\phi^*)$, where we recall that $\phi^*=\phi_{m^*}$, cf.\ Lemma \ref{lem:figureeight} (iv).
  In addition, if $n>2$, then we may also allow the condition $(\alpha,\alpha,\beta)$ for any $\beta\in(0,2\pi-2\alpha)$, where junctions are not asymptotically planar but conical; our construction in Lemma \ref{lem:degenerate_network} also covers this case since we may just replace $P_2\gamma_\mathrm{wave}^m$ by $R_\theta\gamma_\mathrm{wave}^m$, where $R_\theta$ denotes the rotation matrix around the axis $\gamma_\mathrm{seg}^m$ through a suitable angle $\theta\in(0,\pi)$, to satisfy the desired angle condition.
  Accordingly, a parallel argument implies the same existence result of a minimal nondegenerate network as in Theorem \ref{thm:existence_network} in such generalized classes.

  However, the above upper bound $\alpha<\pi-\phi^*$ ($<\frac{3\pi}{4}$) would not be optimal in view of numerical computations in \cite{DallAcqua2017,DallAcqua2020}.
  In fact, we have $4\sqrt{\varpi^*}\approx21.207$, while a suitable double-bubble $\Gamma_\alpha$ corresponding to a given angle $\alpha\in(0,\pi)$ has energy $E[\Gamma_\alpha]=4\sqrt{2\alpha(2\alpha+\sin\alpha)}$ and in particular $E[\Gamma_{3\pi/4}]\approx20.214$.
  Hence, at least up to $\alpha\leq\frac{3\pi}{4}$, a counterpart of Lemma \ref{lem:degenerate_network} holds so that there still exists a nondegenerate minimal elastic network.
  %This fact suggests that minimizers look more round than those constructed in Lemma \ref{lem:degenerate_network}.
  It would be interesting to ask whether and when minimal elastic networks degenerate.
\end{remark}

\appendix

\section{The parameter of figure-eight elastica}\label{sec:app:figure-eight}

Here we give a proof of Lemma \ref{lem:parameter}.
Our proof is based on the series expansion and provides very simple machinery to estimate the value of the parameter $m^*$ of a figure-eight elastica, which can be as accurate as one needs in principle.

\begin{proof}[Proof of Lemma \ref{lem:parameter}]
  Recall the known series expansions of $E$ and $K$ (e.g.\ found in \cite[17.3.11, 17.3.12]{Abramowitz1992}) that are absolutely convergent for any $m\in(0,1)$:
  \begin{align*}
    K(m) &= \int_0^{\pi/2}\frac{1}{\sqrt{1-m\sin^2\theta}}d\theta = \frac{\pi}{2}\sum_{n=0}^\infty\left(\frac{(2n-1)!!}{(2n)!!}\right)^2m^n,\\
    E(m) &= \int_0^{\pi/2}\sqrt{1-m\sin^2\theta}d\theta = \frac{\pi}{2}\sum_{n=0}^\infty\left(\frac{(2n-1)!!}{(2n)!!}\right)^2\frac{1}{1-2n}m^n.
  \end{align*}
  Letting $f(m) := \frac{2}{\pi}(K(m)-2E(m))$, we deduce from these expansions that
  \begin{equation*}
    f(m) = -1 + \sum_{n=1}^\infty A_nm^n, \quad \text{where}\ A_n:=\left(\frac{(2n-1)!!}{(2n)!!}\right)^2\frac{2n+1}{2n-1}.
  \end{equation*}
  Since $m^*$ is a unique root of the increasing continuous function $f$, it is now sufficient to prove that
  \begin{equation}\label{eq03}
    f(0.75)<0, \quad f(0.85)>0.
  \end{equation}
  To this end we first state general criteria for clarity.
  Let
  \begin{equation*}
    S_N(m):=\sum_{n=1}^NA_nm^n, \quad T_N(m):=S_N(m)+\frac{1}{1-m}m^{N+1}.
  \end{equation*}
  Then the following general estimate holds for any $m\in(0,1)$ and any integer $N\geq1$:
  \begin{equation}\label{eq04}
    S_N(m) \leq f(m) + 1 \leq T_N(m).
  \end{equation}
  Indeed, the lower bound is trivial since $A_nm^n\geq0$ and $S_N(m)$ is a partial sum of $\sum A_nm^n$, while the upper bound also follows since $A_n\leq1$ holds by induction and hence $\sum A_nm^n \leq S_N(m) + \sum_{n=N+1}^\infty m^n = T_N(m)$.
  We thus obtain general criteria to compare $f(m)$ and $0$; namely, by \eqref{eq03} and \eqref{eq04} it is now sufficient to find some examples of $N,N'\geq1$ such that $T_{N}(\frac{75}{100})<1$ and $S_{N'}(\frac{85}{100})>1$.
  An explicit computation (of finite operations multiplying integers) shows that
  \begin{equation*}
    T_{10}\left(\frac{75}{100}\right)=\frac{71740047753969831}{72057594037927936}<1, \quad S_7\left(\frac{85}{100}\right)= \frac{1739865847127}{1717986918400} >1,
  \end{equation*}
  thus completing the proof.
\end{proof}

\section{Fenchel's theorem for curves with vertices}\label{sec:app:Fenchel}

In this section we prove Lemma \ref{lem:piecewise}.
Fenchel's theorem asserts that the total curvature $TC[\gamma]=\int_\gamma|\kappa|ds$ satisfies $TC[\gamma]\geq2\pi$ for smooth closed curves in $\mathbf{R}^n$, and hence by the standard mollifier approximation the same is true for $W^{2,1}$ closed curves.
Here we extend this to curves with vertices, or in other words to piecewise $W^{2,1}$ closed curves.
This generalizes \cite[Theorem A.1]{DallAcqua2020} from $n=2$ to $n\geq2$ by a different approach.

Our argument is based on the intuitive understanding that the external angle $\theta$ can be directly incorporated into the total curvature $TC$.
Indeed, if we consider a planar polygon with vertices of external angle $\theta_1,\dots,\theta_N\in(0,\pi)$, then by rounding each vertex by a small circular arc of angle $\theta_j$, we obtain a closed curve $\gamma$ of class $C^{1,1}=W^{2,\infty}$ such that $TC[\gamma]=\theta_1+\dots+\theta_N$.
Our proof is devoted to extending this idea to the general case.
(Another possible approach might be to use Sullivan's FTC-framework \cite{Sullivan2008}.)

\begin{proof}[Proof of Lemma \ref{lem:piecewise}]
  Up to reparameterizations we may assume that $|\gamma_j|\equiv 1$ for all $j$; in this case the domain intervals are of the form $[0,L_j]$ (not necessarily $L_j=1$).
  Without loss of generality we may assume that $\theta_j>0$ for all $j$; indeed, if $\theta_j=0$, then we may regard the composite of $\gamma_j$ and $\gamma_{j+1}$ as a single $W^{2,1}$-curve.
  Below we first consider the generic case that $\max_j\theta_j<\pi$, and then separately treat the case where $\theta_j=\pi$ occurs.

  {\em Case 1: $\theta_j<\pi$ for all $j$.}
  The proof of this case consists of two steps: We first argue for locally polygonal vertices, and then extend it to the general case by approximation.

  Suppose that at all vertices $p_j$, both $\gamma_j$ and $\gamma_{j+1}$ are (locally) given by straight segments in small neighborhoods of the corresponding endpoints.
  In this case, around each vertex $p_j$, the composite of two curves $\gamma_j$ and $\gamma_{j+1}$ makes a planar edge with a vertex of external angle $\theta_j<\pi$, and hence we can concretely round each angular part by using circular arcs $C_j$ of small radius and central angle $\theta_j$.
  As a result we obtain a closed $W^{2,1}$-curve $\bar{\gamma}$.
  Since all the segment parts have zero total curvature, while $TC[C_j]=\theta_j$, we have
  $$TC[\bar{\gamma}]=\sum_{j=1}^NTC[\gamma_j]+\sum_{j=1}^N\theta_j.$$
  On the other hand, Fenchel's theorem for closed $W^{2,1}$-curves implies that $TC[\bar{\gamma}]\geq2\pi$.
  Combining these two relations completes the proof in this case.

  Finally we prove the general case.
  In view of the above ``locally polygonal'' case, it is sufficient to show that for each $j$ there exists an immersed curve $\gamma_{j,\delta}\in W^{2,1}(0,L_j;\mathbf{R}^n)$ such that the family $\gamma_{1,\delta},\dots,\gamma_{N,\delta}$ not only has the same vertices $p_1,\dots,p_N$ and angles $\theta_1,\dots,\theta_N$ as $\gamma_1,\dots,\gamma_N$, but also satisfies the above ``locally polygonal'' assumption, and also $\lim_{\delta\to+0}TC[\gamma_j^\delta]= TC[\gamma_j]$.
  It is sufficient to prove that such a procedure can be done in a small neighborhood of each endpoint independently.
  Let $\gamma\in W^{2,1}(0,L;\mathbf{R}^n)\subset C^1([0,L];\mathbf{R}^n)$ be a given unit-speed curve.
  Let $\eta\in C^\infty(\mathbf{R})$ be a nonincreasing cut-off function such that $\eta\equiv1$ on $(-\infty,\frac{1}{2}]$ and $\eta\equiv0$ on $[1,\infty)$.
  Let $\eta_\delta(x):=\eta(x/\delta)$ for $\delta\in(0,\frac{L}{2})$.
  Let $\xi(s):=\gamma(0)+s\gamma'(0)$ for $s\in[0,L]$ and
  $$\gamma_\delta:=(1-\eta_\delta)\gamma+\eta_\delta\xi.$$
  This curve satisfies that $\gamma_\delta(0)=\xi(0)=\gamma(0)$ and $\gamma_\delta'(0)=\xi'(0)=\gamma'(0)$, that $\gamma_\delta''\equiv0$ on $[0,\delta/2]$ (i.e., locally segment near $s=0$), and that $\gamma_\delta\equiv\gamma$ on $[\delta,L]$.
  It remains to show that $TC[\gamma_\delta]\to TC[\gamma]$.
  It suffices to show that $\gamma_\delta\to\gamma$ in $W^{2,1}$ (since this also implies $C^1$-convergence).
  Here we explicitly compute only the second-order derivative, which is the most delicate part.
  Noting that the $C^1$-curve $\gamma$ can be expanded as $\gamma(s)=\gamma(0)+s\gamma'(0)+o(s)$ and $\gamma'(s)=\gamma'(0)+o(1)$ so that $\gamma(s)-\xi(s)=o(s)$ and $\gamma'(s)-\xi'(s)=o(1)$, we compute
  \begin{align*}
    \gamma_\delta''(s) &= \gamma''(s) - \eta_\delta(s)\gamma''(s) +\eta_\delta'(s)o(1) + \eta_\delta''(s)o(s)\\
    &= \gamma''(s) - \eta(s/\delta)\gamma''(s) + \delta^{-1}\eta'(s/\delta)o(1) + \delta^{-2}\eta''(s/\delta)o(s),
  \end{align*}
  and hence, noting that $\eta^{(i)}(s/\delta)\equiv0$ for $s\geq \delta$ and $i=0,1,2$, we obtain
  \begin{align*}
    \int_0^L|\gamma_\delta''(s)-\gamma''(s)|ds &\leq \int_0^L|\eta(s/\delta)\gamma''(s) + \delta^{-1}\eta'(s/\delta)o(1) + \delta^{-2}\eta''(s/\delta)o(s)|ds\\
    &\leq \int_0^{\delta}|\gamma''(s)|ds + \|\eta'\|_\infty o(1) + \delta^{-1}\|\eta''\|_\infty o(\delta) \to0 \quad (\delta\to0).
  \end{align*}
  Showing similarly that $\gamma'_\delta\to\gamma'$ and $\gamma_\delta\to\gamma$ in $L^1$, we complete the proof.

  \emph{Case 2: $\theta_j=\pi$ for some $j$.}
  We may assume that $\theta_j=\pi$ holds for exactly one $j$ since otherwise the assertion becomes trivial.
  To the vertex $p_j$ we attach a teardrop-like planar $W^{2,1}$-curve $\eta^\delta:[0,L_\delta]\to\mathbf{R}^n$ such that $\eta^\delta(0)=\eta^\delta(L_\delta)=p_j$, $\gamma_j'(L_j)=(\eta^\delta)'(0)=-(\eta^\delta)'(L_\delta)$ ($=-\gamma_{j+1}'(0)$), and $\lim_{\delta\to0}TC[\eta^\delta]=\pi$.
  Such a curve can be constructed by closing a semi-circle of radius $\delta$ by adding almost straight curves; see \cite[Lemma 2.2]{Miura2022} for more details.
  Then we may regard the composite of $\gamma_j,\eta_\delta,\gamma_{j+1}$ as a single $W^{2,1}$-curve $\tilde{\gamma}_j^\delta$ with total curvature $TC[\tilde{\gamma}_j^\delta]=TC[\gamma_j]+(\pi+o(1))+TC[\gamma_{j+1}]$ as $\delta\to0$.
  Therefore, applying Case 1 to the family $\gamma_1,\dots,\gamma_N$ with $\gamma_j,\gamma_{j+1}$ replaced by $\tilde{\gamma}_j^\delta$, and taking $\delta\to0$, we complete the proof.
\end{proof}

\bibliography{bibliography}

\end{document}